\documentclass[3p,10pt,a4paper,twoside,fleqn,sort&compress]{filomat}
\usepackage{amssymb,amsmath,latexsym}
\usepackage[varg]{pxfonts}

\newtheorem{theorem}{Theorem}[section]

\newtheorem{corollary}[theorem]{Corollary}

\newtheorem{definition}[theorem]{Definition}
\newtheorem{example}[theorem]{Example}

\newtheorem{lemma}[theorem]{Lemma}

\newtheorem{proposition}[theorem]{Proposition}

\newcommand{\N}{\mathbb{N}}
    \newcommand{\R}{\mathbb{R}}
    \newcommand{\C}{\mathbb{C}}
    
\allowdisplaybreaks[4]

\newcommand{\FilVol}{xx}
\newcommand{\FilYear}{yyyy}
\newcommand{\FilPages}{zzz--zzz}
\newcommand{\FilDOI}{(will be added later)}
\newcommand{\FilReceived}{Received: 26 March 2020; Accepted: 04 May 2020}

\begin{document}

%

\title{Lyapunov functions for fractional order $h$-difference systems}

\author[affil1,affil2]{Xiang Liu}
\ead{xliu@hebtu.edu.cn}
\author[affil1]{Baoguo Jia}
\ead{mcsjbg@mail.sysu.edu.cn}
\author[affil3]{Lynn Erbe}
\ead{lerbe@unl.edu}
\author[affil3]{Allan Peterson}
\ead{apeterson1@math.unl.edu}
\address[affil1]{School of Mathematics, Sun Yat-Sen
University, Guangzhou, 510275, China}
\address[affil2]{School of Mathematical Sciences, Hebei Normal University, Shijiazhuang, 050024, China}
\address[affil3]{Department of Mathematics, University of Nebraska-Lincoln, Lincoln, NE 68588-0130, U.S.A.}
\newcommand{\AuthorNames}{X. Liu, B. G. Jia, L. Erbe, A. Peterson}

\newcommand{\FilMSC}{39A12; 39A70}
\newcommand{\FilKeywords}{Fractional order $h$-difference systems,
stability, direct method, Lyapunov functions, Euler gamma function}
\newcommand{\FilCommunicated}{Maria Alessandra Ragusa}
\newcommand{\FilSupport}{Research supported by the Nature Science
Foundation of Guangdong Province (No. 2019A1515010609),the Guangdong Province Key Laboratory of Computational Science,
the Nature Science Foundation of Hebei Province (No. A2020205026),
the Scientific Research Foundation of Hebei Education Department (No. QN2020202),
the Science Foundation of Hebei Normal University (No. L2020B01),
and the International Program for Ph.D. Candidates, Sun Yat-Sen University.}

\begin{abstract}
This paper presents some new propositions related to the fractional order $h$-difference operators, for the case of general quadratic forms and for the polynomial type, which allow proving the stability
of fractional order $h$-difference systems, by means of the discrete fractional Lyapunov direct method, using general quadratic Lyapunov functions, and polynomial Lyapunov functions of any positive integer order, respectively. Some examples are given to illustrate these results.
\end{abstract}

\maketitle

\makeatletter
\renewcommand\@makefnmark%
{\mbox{\textsuperscript{\normalfont\@thefnmark)}}}
\makeatother

\section{Introduction}
Fractional calculus, the study of integrals and differences of any order, is a topic of growing interest. Basic information on fractional calculus, concepts, ideas and their applications can be found e.g. in \cite{ks2016,p1999,k2011,smd2018,pr2008}. Dynamical systems are one of the most active areas because of their applications in various fields of science and engineering, and many authors have focused on the stability of the nonlinear fractional systems, see for instance \cite{lc2009,lc2010,yh2013,zl2011,wm2015,cm2014,dc2015,fn2017,
zh2014,bz2015,wb2017,bw2017,ljep2019}.
Due to the lack of geometrical interpretation of the fractional derivatives and the differences, it is difficult to find a valid tool to analyze the stability of the fractional equations, and to our knowledge there are few works on the stability of solutions for either fractional differential equations, see \cite{lc2009,lc2010,yh2013,zl2011} or fractional difference equations, see \cite{cm2014,dc2015,fn2017,zh2014,bz2015,wb2017,bw2017,ljep2019}.

This paper will focus on the stability of the fractional order $h$-difference systems. The fractional order $h$-difference operators are special operators of the  discrete fractional calculus,
which were initiated by Miller and Ross \cite{mr1988} in 1988.
Since then there has been a great deal of interest in the discrete fractional calculus. The basic theory of the discrete fractional calculus can be found in \cite{g2014,g2009,gp2014,j2014,ff2014,ae2011,c2010,wp2015,mw2015,mg2013,mgw2013,jl2017,bft2011} and other sources. The calculus of fractional $h$-differences was given for instance in \cite{wp2015,mw2015,mg2013,mgw2013,jl2017,bft2011}.

So, in order to prove the stability of the fractional order $h$-difference systems, it is not an easy task to directly extend the normal Lyapunov stability results to the fractional cases
since the Leibniz law is complicated and does not hold generally.
Matignon \cite{m1998} proposes an explicit stability condition for a linear fractional differential systems.
The articles \cite{lc2009,lc2010} present the fractional Lyapunov direct method to the fractional order differential systems,
for the applications of this method, see \cite{yh2013,zl2011,wm2015}.
However, it is a difficult task to find an appropriate Lyapunov function by means of this method. Some authors have proposed Lyapunov functions to prove the stability of the fractional order systems.
For the application of this method, we refer to \cite{cm2014,wb2017,bw2017,dc2015,fn2017,zh2014,bz2015,ljep2019}.

In \cite{ljep2019}, we consider the stability of nabla
$(q,h)$-fractional difference equations by Lyapunov direct method
combining with the comparison lemma. In our present paper,
we use a different method to show these comparison lemmas, that allow using general quadratic Lyapunov functions to analyse the stability of the fractional order $h$-difference systems.
We also give some new propositions for the fractional order $h$-difference systems, which enable us to build polynomial Lyapunov functions of any order to determine the stability of such systems.
As a consequence, we give a sufficient condition
for these systems to be stable or asymptotically stable.
Finally, some examples are given to illustrate
our main results.

\section{Preliminary Definitions}
Let $(h\N)_{a}:=\{a,a+h,a+2h,\cdots\}$, where $h>0$, $a\in\R$. We use the notation $\mathcal{F}_{(h\N)_{a}}$ denotes the set of real valued functions defined on $(h\N)_{a}$.
Let $\sigma(t)=t+h$ for $t\in(h\N)_{a}$.
For the convenience of the readers,
we will list some relevant results here.
\begin{definition}(See \cite[Definition 2.1]{mg2013}).\label{d2.1}
For a function $y\in \mathcal{F}_{(h\N)_{a}}$, the forward $h$-difference operator is defined as
\begin{equation}\label{2.1}
(\Delta_{h}y)(t)=\frac{y(\sigma(t))-y(t)}{h},\ \ t\in (h\N)_{a}.
\end{equation}
and the $h$-difference sum is given by
\begin{equation*}
(_{a}\Delta_{h}^{-1}y)(t)=\sum_{s=\frac{a}{h}}^{\frac{t}{h}-1}y(sh)h,\ \ t\in(h\N)_{a},
\end{equation*}
where, by convention, $(_{a}\Delta_{h}^{-1}y)(a)=0$.
\end{definition}

\begin{definition}(See \cite[Definition 2.6]{bft2011}).\label{d2.2}
For arbitrary $t$, $\nu\in \R$, the $h$-factorial function is defined by
\begin{equation}\label{2.2}
t_{h}^{(\nu)}=h^{\nu}\frac{\Gamma(\frac{t}{h}+1)}{\Gamma(\frac{t}{h}+1-\nu)},
\end{equation}
where $\Gamma$ is the Euler gamma function with $\frac{t}{h}+1 \notin Z_{-}\cup\{0\}$,
and we use the convention that $t_{h}^{(\nu)}=0$, when $\frac{t}{h}+1-\nu$ is a nonpositive integer,
and $\frac{t}{h}+1$ is not a nonpositive integer.
\end{definition}

\begin{definition}(See \cite[Definition 2.8]{bft2011}).\label{d2.3}
For a function $y\in \mathcal{F}_{(h\N)_{a}}$, the fractional $h$-sum of order $\nu>0$
is given by
\begin{equation}\label{2.3}
(_{a}\Delta_{h}^{-\nu}y)(t)=\frac{h}{\Gamma(\nu)}
\sum_{s=\frac{a}{h}}^{\frac{t}{h}-\nu}
(t-\sigma(sh))_{h}^{(\nu-1)}y(sh),\ \ t\in (h\N)_{a+\nu h},
\end{equation}
and $(_{a}\Delta_{h}^{0}y)(t)=y(t)$, $\sigma(sh)=(s+1)h$.
\end{definition}

\begin{definition}(See \cite[Definition 2.6]{mg2013}).\label{d2.4}
Let $\nu\in(n-1,n]$, and set $\mu=n-\nu$, where $n\in \N_{1}$.
The Riemann-Liouville like fractional $h$-difference operator $_{a}\Delta_{h}^{\nu}y$ of order $\nu$
for a function $y\in \mathcal{F}_{(h\N)_{a}}$ is defined by
\begin{equation}\label{2.4}
(_{a}\Delta_{h}^{\nu}y)(t)=(\Delta_{h}^{n}(_{a}\Delta_{h}^{-\mu}y))(t)
=\frac{h}{\Gamma(\mu)}\Delta_{h}^{n}\sum_{s=\frac{a}{h}}^{\frac{t}{h}-\mu}
(t-\sigma(sh))_{h}^{(\mu-1)}y(sh),\ \ t\in (h\N)_{a+\mu h}.
\end{equation}
\end{definition}

\begin{lemma}(See \cite[Theorem 2.8]{mg2013}).\label{l2.1}
Let $\nu\in (n-1,n]$, and $\mu=n-\nu$, where $n\in \N_{1}$.
The following formula is equivalent to \eqref{2.4}:
\begin{equation}\label{2.5}
\arraycolsep=1pt
 \begin{array}{ll}
(_{a}\Delta_{h}^{\nu}y)(t)=
\left\{
  \begin{aligned}
& \frac{h}{\Gamma(-\nu)}\sum_{s=\frac{a}{h}}^{\frac{t}{h}+\nu}
(t-\sigma(sh))_{h}^{(-\nu-1)}y(sh),\ \ \nu\in(n-1,n),\  t\in (h\N)_{a+\mu h},\\
& (\Delta_{h}^{n}y)(t),\ \ \nu=n, \ t\in (h\N)_{a}.
 \end{aligned}
\right.
\end{array}
\end{equation}
 \end{lemma}

\begin{lemma}(See \cite[Definition 2.9]{mg2013}).\label{l2.2}
Let $\nu\in(n-1,n]$, and set $\mu=n-\nu$, where $n\in\N_{1}$.
The Caputo like $h$-difference operator
$_{a}\Delta_{h,\ast}^{\nu}y$ of order $\nu$
for a function $y\in\mathcal{F}_{(h\N)_{a}}$ is defined by
\begin{equation}\label{2.6}
(_{a}\Delta_{h,\ast}^{\nu}y)(t)=
(_{a}\Delta_{h}^{-\mu}(\Delta_{h}^{n}y))(t)
=\frac{h}{\Gamma(\mu)}\sum_{s=\frac{a}{h}}^{\frac{t}{h}-\mu}
(t-\sigma(sh))_{h}^{(\mu-1)}(\Delta_{h}^{n}y)(sh), \ \ t\in(h\N)_{a+\mu h}.
\end{equation}
\end{lemma}
\begin{lemma}(See \cite[Proposition 2.11]{mg2013}).\label{l2.3}
Let $\nu\in(n-1,n]$, and $\mu=n-\nu$, where $n\in\N_{1}$.
The following formula is equivalent to \eqref{2.6}:
\begin{equation}\label{2.7}
(_{a}\Delta_{h,\ast}^{\nu}y)(t)=
\arraycolsep=1pt
 \begin{array}{ll}
\left\{
  \begin{aligned}
& \frac{h^{1-n}}{\Gamma(n-\nu)}\sum_{s=\frac{a}{h}}^{\frac{t}{h}-(n-\nu)}
(t-\sigma(sh))_{h}^{(n-\nu-1)}\sum_{r=0}^{n}(-1)^{r+1}
\left( {\begin{array}{*{20}{c}}
n\\
r
\end{array}} \right)y((r+s)h),\ \ \nu\in(n-1,n),\ t\in(h\N)_{a+\mu h},\\
& (\Delta_{h}^{n}y)(t),\ \ \nu=n,\ t\in(h\N)_{a}.
 \end{aligned}
\right.
\end{array}
\end{equation}
\end{lemma}
The following corollary appears in Mozyrska et al \cite[Corollary 4.2]{mg2013}.

\begin{corollary}(See \cite[Corollary 4.2]{mg2013}).\label{c2.1}
Let $\nu\in(0,1]$. The following formula holds
\begin{equation}\label{2.8}
(_{a}\Delta_{h,\ast}^{\nu}y)(t)=(_{a}\Delta_{h}^{\nu}y)(t)
-\frac{(t-a)_{h}^{(-\nu)}y(a)}{\Gamma(1-\nu)},
\ \ t\in (h\N)_{a+(1-\nu) h}.
\end{equation}
\end{corollary}

\begin{lemma}\label{p2.1}
The following properties are useful in this paper:

Delta difference of the $h$-falling factorial function
(See \cite[Lemma 3.2]{jl2017})
\begin{equation}\label{2.9}
_{s}\Delta_{h}(t-sh)_{h}^{(\nu)}=-\nu(t-\sigma(sh))_{h}^{(\nu-1)},
\end{equation}
where $_{s}\Delta_{h}(t-sh)_{h}^{(\nu)}
=\frac{(t-sh-h)_{h}^{(\nu)}-(t-sh)_{h}^{(\nu)}}{h}$.

Summation by parts (See \cite[Property 2.3]{bw2017})
\begin{equation}\label{2.10}
\sum_{s=\frac{a}{h}}^{\frac{t}{h}+\nu-1}x(sh+h)(\Delta_{h} y)(sh)=
\frac{1}{h}(x(sh)y(sh))\big|_{s=\frac{a}{h}}^{\frac{t}{h}+\nu}
-\sum_{s=\frac{a}{h}}^{\frac{t}{h}+\nu-1}y(sh)(\Delta_{h} x)(sh),\ \ t\in (h\N)_{a+(1-\nu)h}.
\end{equation}
\end{lemma}

\section{General quadratic Lyapunov functions for stability}
In this section, we will demonstrate stability of
the fractional order $h$-difference systems
by finding a general quadratic Lyapunov functions,
using the discrete fractional Lyapunov direct method.

Consider the following nonlinear vector fractional order $h$-difference equations
 \begin{equation}\label{3.1}
\left\{
  \begin{array}{ll}
(_{a}\Delta_{h,\ast}^{\nu} x)(t)= f(t,x(t+\nu h)),\ \ t \in (h \N)_{a+(1-\nu)h},\\
x(a)=x_{0}\in \R^{n},
\end{array}
\right.
\end{equation}
and
\begin{equation}\label{3.2}
\left\{
  \begin{array}{ll}
(_{a}\Delta_{h}^{\nu} x)(t)= f(t,x(t+\nu h)),\ \ t \in (h \N)_{a+(1-\nu)h},\\
(_{a}\Delta_{h}^{\nu-1} x)(t)|_{t=a+(1-\nu)h}
=h^{1-\nu}x_{0}\in \R^{n},
\end{array}
\right.
\end{equation}
where $f: (h\N)_{a+(1-\nu)h}\times \R^{n}\to \R^{n}$
is continuous with respect to $x$,
$x: (h\N)_{a}\to \R^{n}$,
 and $\nu\in(0,1]$. By Schauder's Fixed Point Theorem,
 it is easy to show that equations
\eqref{3.1} and \eqref{3.2} has a solution which exists in $(h\N)_{a}$,
we could refer to \cite[Theorem 6.11]{gp2014}.

The constant vector $x_{eq}$ is an \it equilibrium point \rm of the nonlinear vector fractional
order $h$-difference equation
\eqref{3.1} (or \eqref{3.2}) if and only if $(_{a}\Delta_{h,\ast}^{\nu} x_{eq})(t)= f(t,x_{eq}(t+\nu h))=0$
($(_{a}\Delta_{h}^{\nu} x_{eq})(t)= f(t,x_{eq}(t+\nu h))$ in the case of the Riemann-Liouville
vector fractional $h$-difference equation) for all
$t\in (h\N)_{a+(1-\nu)h}$.

Assume that $f(t,0)=0$ so that the trivial solution
$x\equiv 0$ is an equilibrium point of
the fractional order $h$-difference
system \eqref{3.1} (or \eqref{3.2}).
Note that there is no loss of generality in doing so
because any equilibrium point can be shifted to
the origin via a change of variables.

First, we present the following simple definitions and important facts.

 \begin{definition}\label{d3.1}
 The equilibrium point $x=0$ of
 the system \eqref{3.1} (or \eqref{3.2}) is said to be

 (a) stable, if for each $\varepsilon>0$, there exists $\delta=\delta(\varepsilon)>0$ such that
 $\|x(a)\|<\delta$ implies $\|x(a+kh)\|<\varepsilon$ for all $k\in \N_{0}$.

 (b) attractive, if there exists $\delta>0$ such that $\|x(a)\|<\delta$ implies
 $\lim_{k\to\infty}x(a+kh)=0$.

 (c) asymptotically stable, if it is stable and attractive.
 \end{definition}

The fractional order $h$-difference system \eqref{3.1} (or \eqref{3.2}) is called stable (asymptotically stable)
if their equilibrium point $x=0$ is stable (asymptotically stable).

\begin{definition}
Let $V: (h\N)_{a}\times\mathbb{R}^n \to \mathbb{R}$
be a continuous scalar function.
$V$ is a Lyapunov function if it is a locally positive-definite function, i.e.
$V(t,0) = 0$,
$V(t,x) > 0$, $\forall x \in U\setminus\{0\}$
with $U$ being a neighborhood region around $x = 0$.
\end{definition}

\begin{definition}(See \cite[Definition 3.2]{ja2012}).\label{d3.2}
 A function $\phi(r)$ is said to belong to the class $\mathcal{K}$
 if and only if $\phi\in C[[0,\rho),\R_{+}]$,
 $\phi(0)=0$, and $\phi(r)$ is strictly monotonically increasing in $r$.
\end{definition}

\begin{definition}\label{d3.3}
A real valued function $V(t,x)$ defined on
$(h\N)_{a}\times S_{\rho}$, where
$S_{\rho}=\{x\in \R^{n}:\|x\|\leq \rho\}$, is said to be positive definite if and only if
$V(t,0)=0$ for all $t\in (h\N)_{a}$ and there exists $\phi(r)\in \mathcal{K}$
 such that $\phi(r)\leq V(t,x)$, $\|x\|=r$,
 $(t,x)\in (h\N)_{a}\times S_{\rho}$.
\end{definition}

\begin{definition}\label{d3.4}
A real valued function $V(t,x)$ defined on $(h\N)_{a}\times S_{\rho}$, where $S_{\rho}=\{x\in \R^{n}:\|x\|\leq \rho\}$,
is said to be decrescent if and only if $V(t,0)=0$ for all $t\in (h\N)_{a}$ and there exists
$\phi(r)\in \mathcal{K}$ such that $V(t,x)\leq\phi(r)$, $\|x\|=r$, $(t,x)\in (h\N)_{a}\times S_{\rho}$.
\end{definition}

Now, let us recall that the $\mathcal{Z}$-transform of a sequence $\{y(n)\}_{n\in \N_{0}}$
is a complex function given by $Y(z)=\mathcal{Z}[y](z)=\sum_{k=0}^{\infty}\frac{y(k)}{z^{k}}$,
where $z\in \C$ is a complex number for
which this series converges absolutely.
The $\mathcal{Z}$-transform of
$\widetilde{\varphi}_{\alpha}$ is defined as
\begin{equation*}
\mathcal{Z}[\widetilde{\varphi}_{\alpha}](z)=
\sum_{k=0}^{\infty}\left( {\begin{array}{*{20}{c}}
k+\alpha-1\\
k
\end{array}} \right)\frac{1}{z^{k}}
=\sum_{k=0}^{\infty}(-1)^{k}\left( {\begin{array}{*{20}{c}}
-\alpha\\
k
\end{array}} \right)\frac{1}{z^{k}}=\Big(\frac{z}{z-1}\Big)^{\alpha},
\end{equation*}
where
$$
\tilde{\varphi}_{\alpha}(n)=\left( {\begin{array}{*{20}{c}}
n+\alpha-1\\
n
\end{array}} \right)=(-1)^{n}\left( {\begin{array}{*{20}{c}}
-\alpha\\
n
\end{array}} \right).
$$
The convolution of $\tilde{\varphi}_{\alpha}$ and $x$ is defined as
\begin{equation*}
(\tilde{\varphi}_{\alpha}\ast x)(n)=\sum_{s=0}^{n}\left( {\begin{array}{*{20}{c}}
n-s+\alpha-1\\
n-s
\end{array}} \right)x(s).
\end{equation*}
The following relations hold for the $\mathcal{Z}$-transform,
$\mathcal{Z}[x(n)\ast y(n)](z)=\mathcal{Z}[x(n)](z)\mathcal{Z}[y(n)](z)$,
$\mathcal{Z}[y(n-1)](z)=\frac{1}{z}\mathcal{Z}[y(n)](z)$
for $|z|>R$,
where $R$ is the radius of convergence of $\mathcal{Z}[y(n)](z)$.

\begin{proposition}(See \cite[Proposition 16]{mw2015}).\label{p3.2}
For $\nu\in (0,1]$, $a\in \R$, let $y(n):=(_{a}\Delta_{h,\ast}^{\nu}x)(a+(1-\nu)h+nh)$,
where $n\in\N_{0}$. Then
\begin{equation}\label{3.3}
\mathcal{Z}[y](z)=h^{-\nu}\Big(\frac{z}{z-1}\Big)^{1-\nu}
[(z-1)X(z)-zx(a)],
\end{equation}
where $X(z)=\mathcal{Z}[\bar{x}](z)$,
and $\bar{x}(n)=x(a+nh)$.
\end{proposition}

\begin{proposition}(See \cite[Proposition 24]{mw2015}).\label{p3.3}
For $\nu\in (0,1]$, $a\in \R$, let $y(n):=(_{a}\Delta_{h}^{\nu}x)(a+(1-\nu)h+nh)$,
where $n\in\N_{0}$. Then
\begin{equation}\label{3.4}
\mathcal{Z}[y](z)=z\Big(\frac{hz}{z-1}\Big)^{-\nu} X(z)
-zh^{-\nu}x(a),
\end{equation}
where $X(z)=\mathcal{Z}[\bar{x}](z)$, and $\bar{x}(n)=x(a+nh)$.
\end{proposition}

\begin{lemma}(Diagonalization of a real symmetric matrix
\cite[p. 54]{r1960}).\label{l3.1}
Let $P\in \R^{n\times n}$ be a real symmetric matrix.
Then it may be transformed into
a diagonal form by means of an orthogonal transformation,
that is to say,
there is an orthogonal matrix $B\in \R^{n\times n}$ and
a diagonal matrix $\Lambda\in \R^{n\times n}$ such that
\begin{equation*}
P=B\Lambda B^{T}.
\end{equation*}
\end{lemma}

Now, we give the following lemmas for
the Caputo fractional order $h$-difference,
which will be useful for proving the stability of the system \eqref{3.1}.

\begin{lemma}\label{l3.2}
Assume $\nu\in (0,1]$, $x$, $y\in\R$, $(_{a}\Delta_{h,\ast}^{\nu}x)(t)\geq(_{a}\Delta_{h,\ast}^{\nu}y)(t)$,
 $t\in (h\N)_{a+(1-\nu)h}$.
Then we have
$x(t)-y(t)\geq x(a)-y(a)$ for $t\in(h\N)_{a}$.
\end{lemma}

\begin{proof}
It follows from $(_{a}\Delta_{h,\ast}^{\nu}x)(t)\geq(_{a}\Delta_{h,\ast}^{\nu}y)(t)$,
 that there exists a nonnegative function $M(\cdot)$ satisfying
 \begin{equation}\label{3.5}
(_{a}\Delta_{h,\ast}^{\nu}x)(t)=(_{a}\Delta_{h,\ast}^{\nu}y)(t)+M(t).
\end{equation}
Taking the $\mathcal{Z}$-transform of the equation \eqref{3.5},
and using the Proposition \ref{p3.2}, we get
 \begin{equation*}
 h^{-\nu}\Big(\frac{z}{z-1}\Big)^{1-\nu}[(z-1)X(z)-zx(a)]=h^{-\nu}
 \Big(\frac{z}{z-1}\Big)^{1-\nu}[(z-1)Y(z)-zy(a)]+M(z),
 \end{equation*}
 where $X(z)=\mathcal{Z}[x](z)$, $Y(z)=\mathcal{Z}[y](z)$.
It follows that
 \begin{equation}\label{3.6}
X(z)= Y(z)+h^{\nu}\frac{1}{z}\Big(\frac{z}{z-1}\Big)^{\nu}M(z)
+\frac{z}{z-1}(x(a)-y(a)).
\end{equation}
Applying the inverse $\mathcal{Z}$-transform to
the equation \eqref{3.6} gives
\begin{equation*}
x(t)=y(t)+h^{\nu}(\widetilde{\varphi}_{\nu}\ast M)(n-1)+(x(a)-y(a)),
\end{equation*}
where $\widetilde{\varphi}_{\nu}(n)=
\left( {\begin{array}{*{20}{c}}
{n + \nu  - 1}\\
n
\end{array}} \right)=
(-1)^{n}\left( {\begin{array}{*{20}{c}}
{- \nu  }\\
n
\end{array}} \right)$.
It follows from $M(t)\geq 0$ and
$h^{\nu}(\widetilde{\varphi}_{\nu}\ast M)(n-1)\geq 0$ that
$x(t)-y(t)\geq x(a)-y(a)$. The proof is complete.
\end{proof}

\begin{proposition}(See \cite[Lemma 3.2]{bw2017}).\label{l3.3}
For $\nu\in(0, 1]$, and $y\in\R$,
the following inequality holds
\begin{equation}\label{3.7}
(_{a}\Delta_{h,\ast}^{\nu}y^{2})(t)\leq 2y(t+\nu h)
(_{a}\Delta_{h,\ast}^{\nu}y)(t),\ \ t\in (h\N)_{a+(1-\nu)h}.
\end{equation}
\end{proposition}
 \begin{proposition}\label{l3.4}
For $\nu\in (0,1]$, and $y\in\R^{n}$,
the following relationship holds
 \begin{equation}\label{3.8}
\frac{1}{2} {_{a}\Delta_{h,\ast}^{\nu}}(y^{T}(t)Py(t))\leq
y^{T}(t+\nu h)P (_{a}\Delta_{h,\ast}^{\nu}y)(t),\ \ t\in (h\N)_{a+(1-\nu)h},
\end{equation}
where $P\in \R^{n}\times \R^{n}$ is a constant, square,
symmetric, and positive definite matrix.
 \end{proposition}

 \begin{proof}
Since the matrix $P$ is symmetric, using Lemma \ref{l3.1},
there exists an orthogonal matrix $B\in \R^{n\times n}$,
and a diagonal matrix $\Lambda\in \R^{n\times n}$ such that
\begin{equation*}
\frac{1}{2}y^{T}(t)P y(t)=\frac{1}{2}y^{T}(t)B\Lambda B^{T} y(t)
=\frac{1}{2}(B^{T}y(t))^{T}\Lambda(B^{T}y(t)).
\end{equation*}
Set $x(t):=B^{T}y(t)$, then we have
\begin{equation}\label{3.9}
\frac{1}{2} {_{a}\Delta_{h,\ast}^{\nu}}(x^{T}(t)\Lambda x(t)) =
\frac{1}{2} (_{a}\Delta_{h,\ast}^{\nu}\sum_{i=1}^{n}
\lambda_{ii}x_{i}^{2})(t)
=\sum_{i=1}^{n}\lambda_{ii}\frac{1}{2}
(_{a}\Delta_{h,\ast}^{\nu}x_{i}^{2})(t),
\end{equation}
Applying Proposition \ref{l3.3} to the equality \eqref{3.9},
and remembering the positivity hypothesis of the matrix $P$ we have that $\lambda_{ii}>0$, then
\begin{equation*}
\frac{1}{2} {_{a}\Delta_{h,\ast}^{\nu}}(x^{T}(t)\Lambda x(t))
\leq\sum_{i=1}^{n}\lambda_{ii} x_{i}(t+\nu h)(_{a}\Delta_{h,\ast}^{\nu}x_{i})(t),
\end{equation*}
that is,
\begin{equation}\label{3.10}
\frac{1}{2} {_{a}\Delta_{h,\ast}^{\nu}}(x^{T}(t)\Lambda x(t))
\leq x^{T}(t+\nu h)\Lambda(_{a}\Delta_{h,\ast}^{\nu}x)(t).
\end{equation}
Then, replacing $x(t)=B^{T}y(t)$ in the inequality \eqref{3.10},
we have
\begin{equation}\label{3.11}
\frac{1}{2} {_{a}\Delta_{h,\ast}^{\nu}}((B^{T}y(t))^{T}\Lambda (B^{T}y(t)))
\leq (B^{T}y(t+\nu h))^{T}\Lambda{_{a}\Delta_{h,\ast}^{\nu}}(B^{T}y(t)).
\end{equation}
Rearranging and using $B\Lambda B^{T}=P$
in the inequality \eqref{3.11}, we get
\begin{equation*}
\frac{1}{2} {_{a}\Delta_{h,\ast}^{\nu}}(y^{T}(t)P y(t))
\leq y^{T}(t+\nu h)P(_{a}\Delta_{h,\ast}^{\nu}y)(t).
\end{equation*}
The proof is complete.
\end{proof}

\begin{lemma}\label{l3.5}
Let $x=0$ be an equilibrium point of the system \eqref{3.1}.
If there exists a positive definite and
decrescent scalar function $V(t,x)$, class-$\mathcal{K}$ functions
$\gamma_{1}$, $\gamma_{2}$, and $\gamma_{3}$ such that
\begin{equation}\label{3.12}
\gamma_{1}(\|x(t)\|)\leq V(t,x(t))\leq \gamma_{2}(\|x(t)\|),\ \ t\in (h\N)_{a},
\end{equation}
and
\begin{equation}\label{3.13}
(_{a}\Delta_{h,\ast}^{\nu}V)(t,x(t))\leq -\gamma_{3}(\|x(t+\nu h)\|).
\end{equation}
Then the system \eqref{3.1} is asymptotically stable.
\end{lemma}

\begin{proof}
From the equations \eqref{3.12} and \eqref{3.13},
we have
\begin{equation*}
(_{a}\Delta_{h,\ast}^{\nu}V)(t,x(t))\leq
-\gamma_{3}(\gamma_{2}^{-1}(V(t+\nu h,x(t+\nu h))),
\end{equation*}
where $\gamma_{2}^{-1}$ denotes the inverse of $\gamma_{2}$.
It is evident that $\gamma_{3}\circ\gamma_{2}^{-1}$
is a discrete class $\mathcal{K}$ function.

Considering a fractional difference equation
\begin{equation*}
(_{a}\Delta_{h,\ast}^{\nu}U)(t,x(t))=
-\gamma_{3}(\gamma_{2}^{-1}(U(t+\nu h,x(t+\nu h))).
\end{equation*}
Similar to the proof of \cite[Theorem 2.10]{bw2017},
we could show that $V(t,x(t))\leq U(t,x(t))$.
Then, as the proof of \cite[Theorem 2.9]{bw2017},
we could obtain $\lim_{t\to \infty}U(t,x(t))=0$.
Since $\gamma_{1}$ is a class $K$ function,
and $\gamma_{1}(\|x(t)\|)\leq V(t,x(t))$,
it follows that $\lim_{t\to\infty}x(t)=0$.
The proof is complete.
\end{proof}

\begin{theorem}\label{t3.1}
Assume $x=0$ is an equilibrium point of the system \eqref{3.1},
if the following condition is satisfied
\begin{equation*}
x^{T}(t+\nu h)P f(t,x(t+vh))\leq 0,\ \ t\in(h\N)_{a+(1-\nu)h},
\end{equation*}
then the system \eqref{3.1} is stable.
Also, if
\begin{equation*}
x^{T}(t+\nu h)P f(t,x(t+vh))< 0,\ \ t\in(h\N)_{a+(1-\nu)h}, \forall x\neq 0,
\end{equation*}
then the system \eqref{3.1} is asymptotically stable.
 \end{theorem}

 \begin{proof}
Let us propose the following Lyapunov function,
which is positive definite
\begin{equation*}
 V(t,x(t))=\frac{1}{2}x^{T}(t)Px(t).
\end{equation*}
Using Proposition \ref{l3.4}, we obtain
\begin{equation}\label{3.14}
(_{a}\Delta_{h,\ast}^{\nu}V)(t)\leq x^{T}(t+\nu h)P(_{a}\Delta_{h,\ast}^{\nu}x)(t)
=x^{T}(t+\nu h)Pf(t,x(t+\nu h))\leq 0,
\end{equation}
 by Lemma \ref{l3.2}, we have
\begin{equation*}
V(t,x(t))\leq V(a,x(a)),
\end{equation*}
that is,
\begin{equation*}
\frac{1}{2}x^{T}(t)Px(t)\leq \frac{1}{2}x^{T}(a)Px(a).
\end{equation*}
Since $B\Lambda B^{T}=P$, we obtain
\begin{equation*}
\frac{1}{2}x^{T}(t)B\Lambda B^{T}x(t)
\leq \frac{1}{2}x^{T}(a)B\Lambda B^{T}x(a).
\end{equation*}
Set $y(t):=B^{T}x(t)$, then we have
\begin{equation*}
\frac{1}{2}y^{T}(t)\Lambda y(t)\leq \frac{1}{2}y^{T}(a)\Lambda y(a).
\end{equation*}
Since $\Lambda$ is a diagonal matrix, it follows that
\begin{equation*}
\sum_{i=1}^{n}\lambda_{ii}y_{i}^{2}(t)\leq \sum_{i=1}^{n}\lambda_{ii}y_{i}^{2}(a).
\end{equation*}
Hence, we have
\begin{equation*}
\lambda_{\min}\|y(t)\|^{2}\leq \lambda_{\max}\|y(a)\|^{2},
\end{equation*}
where $\lambda_{\min}=\min\{|\lambda_{ii}|: 1\leq i\leq n \}$,
$\lambda_{\max}=\max\{|\lambda_{ii}|: 1\leq i\leq n \}$.

Since $B$ is an orthogonal matrix, and $y(t)=B^{T}x(t)$, we have
\begin{equation*}
\lambda_{\min}\|x(t)\|^{2}\leq \lambda_{\max}\|x(a)\|^{2}.
\end{equation*}
According to the definition of stability in the sense of Lyapunov,
we see that the system \eqref{3.1} is stable in the sense of Lyapunov.

If
\begin{equation*}
x^{T}(t+\nu h)P f(t,x(t+vh))< 0,\ \ t\in(h\N)_{a+(1-\nu)h}, \forall x\neq 0,
\end{equation*}
similar to the above step, we can show that
the system \eqref{3.1} is stable.
Using Proposition \ref{l3.4}, we have $(_{a}\Delta_{h,\ast}^{\nu}V)(t,x(t))\leq x^{T}(t+\nu h)P(_{a}\Delta_{h,\ast}^{\nu}x)(t)< 0$,
that is, the fractional order $h$-difference of
$V$ function is negative definite.
Given the relationship between positive definite functions and class-$\mathcal{K}$
functions in \cite{sl1991}.
As a result, the system \eqref{3.1} is asymptotically stable
from Lemma \ref{l3.5}. The proof is complete.
\end{proof}

In what follows, we will present results concerning the Riemann-Liouville fractional order $h$-difference,
which are important to prove the stability of the system \eqref{3.2}.
\begin{lemma}\label{l3.6}
Assume $\nu\in (0,1]$, $x$, $y\in\R$, $(_{a}\Delta_{h}^{\nu}x)(t)\geq(_{a}\Delta_{h}^{\nu}y)(t)$,
 $t\in (h\N)_{a+(1-\nu)h}$, and $x(a)\leq y(a)$.
Then we have
$x(t)-y(t)\geq x(a)-y(a)$ for $t\in(h\N)_{a}$.
\end{lemma}

\begin{proof}
It follows from $(_{a}\Delta_{h}^{\nu}x)(t)\geq(_{a}\Delta_{h}^{\nu}y)(t)$
 that there exists a nonnegative function $M(\cdot)$ satisfying
 \begin{equation}\label{3.15}
(_{a}\Delta_{h}^{\nu}x)(t)=(_{a}\Delta_{h}^{\nu}y)(t)+M(t).
\end{equation}
Taking the $\mathcal{Z}$-transform of the equation \eqref{3.15},
and using the Proposition \ref{p3.3}, we get
 \begin{equation*}
 z\Big(\frac{hz}{z-1}\Big)^{-\nu}X(z)-zh^{-\nu}x(a)=
 z\Big(\frac{hz}{z-1}\Big)^{-\nu}Y(z)-zh^{-\nu}y(a)+M(z).
 \end{equation*}
It follows that
 \begin{equation}\label{3.16}
X(z)= Y(z)+\frac{1}{z}\Big(\frac{hz}{z-1}\Big)^{\nu}M(z)
+\Big(\frac{z}{z-1}\Big)^{\nu}(x(a)-y(a)).
\end{equation}
Applying the inverse $\mathcal{Z}$-transform to the equation \eqref{3.16} gives
\begin{equation*}
x(t)=y(t)+h^{\nu}(\widetilde{\varphi}_{\nu}\ast M)(n-1)+\widetilde{\varphi}_{\nu}(n)(x(a)-y(a)),
\end{equation*}
where $\widetilde{\varphi}_{\nu}(n)=
\left( {\begin{array}{*{20}{c}}
{n + \nu  - 1}\\
n
\end{array}} \right)=
(-1)^{n}\left( {\begin{array}{*{20}{c}}
{- \nu  }\\
n
\end{array}} \right)\leq1$.
It follows from $\widetilde{\varphi}_{\nu}(n)\geq0$, $x(a)\leq y(a)$, $M(t)\geq 0$,
and $h^{\nu}(\widetilde{\varphi}_{\nu}\ast M)(n-1)\geq 0$ that
$x(t)-y(t)\geq x(a)-y(a)$. The proof is complete.
\end{proof}

\begin{proposition}\label{l3.7}
For $\nu\in(0,1]$, and $y\in\R$,
the following inequality holds
\begin{equation}\label{3.17}
(_{a}\Delta_{h}^{\nu}y^{2})(t)\leq 2 y(t+\nu h)(_{a}\Delta_{h}^{\nu}y)(t),\ \ t\in (h\N)_{a+(1-\nu)h}.
\end{equation}
\end{proposition}

\begin{proof}
If $\nu=1$, we have
\begin{equation*}
(_{a}\Delta_{h}^{\nu}y^{2})(t)
=(\Delta_{h}y^{2})(t)=\frac{y^{2}(t+h)-y^{2}(t)}{h},
\end{equation*}
and
\begin{equation*}
\begin{split}
2 y(t+\nu h)(_{a}\Delta_{h}^{\nu}y)(t)&=2 y(t+ h)(\Delta_{h} y)(t)
=\frac{2 y^{2}(t+ h)-2y(t+ h)y(t)}{h}
\geq \frac{y^{2}(t+h)-y^{2}(t)}{h}.
\end{split}
\end{equation*}
So, the inequality \eqref{3.17} holds for $\nu=1$.

If $\nu\in(0,1)$, it follows from Corollary \ref{c2.1} that
\begin{equation*}
(_{a}\Delta_{h,*}^{\nu}y)(t)=(_{a}\Delta_{h}^{\nu}y)(t)
-\frac{(t-a)_{h}^{(-\nu)}y(a)}{\Gamma(1-\nu)}.
\end{equation*}
Then, we have
\begin{equation*}
y(t+\nu h)(_{a}\Delta_{h,*}^{\nu}y)(t)
=y(t+\nu h)(_{a}\Delta_{h}^{\nu}y)(t)-y(t+\nu h)
\frac{(t-a)_{h}^{(-\nu)}y(a)}{\Gamma(1-\nu)},
\end{equation*}
and
\begin{equation*}
(_{a}\Delta_{h,*}^{\nu}y^{2})(t)=(_{a}\Delta_{h}^{\nu}y^{2})(t)
-\frac{(t-a)_{h}^{(-\nu)}y^{2}(a)}{\Gamma(1-\nu)}.
\end{equation*}
Since $_{s}\Delta_{h}y^{2}(t+\nu h)=0$, and using the summation by parts formula, we obtain
\begin{equation*}
\begin{split}
&2y(t+\nu h)(_{a}\Delta_{h,\ast}^{\nu}y)(t)
-(_{a}\Delta_{h,\ast}^{\nu}y^{2})(t)\\
&~~~=-\frac{h}{\Gamma(1-\nu)}\sum_{s=\frac{a}{h}}^{\frac{t}{h}+\nu-1}
(t-\sigma(sh))_{h}^{(-\nu)} {_{s}\Delta_{h}}\big(y(t+\nu h)-y(sh)\big)^{2}\\
&~~~=-\frac{(t-sh)_{h}^{(-\nu)}}{\Gamma(1-\nu)}\big(y(t+\nu h)
-y(sh)\big)^{2}|_{s=\frac{a}{h}}^{\frac{t}{h}+\nu} +\frac{h}{\Gamma(1-\nu)}\sum_{s=\frac{a}{h}}^{\frac{t}{h}+\nu-1}
\big(y(t+\nu h)-y(sh)\big)^{2}{_{s}\Delta_{h}}(t-sh)_{h}^{(-\nu)}\\
&~~~=\frac{(t-a)_{h}^{(-\nu)}}{\Gamma(1-\nu)}\big(y(t+\nu h)
-y(a)\big)^{2}
+\frac{h\nu}{\Gamma(1-\nu)}\sum_{s=\frac{a}{h}}^{\frac{t}{h}+\nu-1}
(t-\sigma(sh))_{h}^{(-\nu-1)}\big(y(t+\nu h)-y(sh)\big)^{2}.
\end{split}
\end{equation*}
Then, we arrive at
\begin{eqnarray*}
&&2y(t+\nu h)(_{a}\Delta_{h}^{\nu}y)(t)
-(_{a}\Delta_{h}^{\nu}y^{2})(t)\\
&&~~~=2y(t+\nu h)(_{a}\Delta_{h,*}^{\nu}y)(t)+2y(t+\nu h)
\frac{(t-a)_{h}^{(-\nu)}y(a)}{\Gamma(1-\nu)} -\Big[(_{a}\Delta_{h,*}^{\nu}y^{2})(t)
+\frac{(t-a)_{h}^{(-\nu)}y^{2}(a)}{\Gamma(1-\nu)}\Big]\\
&&~~~=2y(t+\nu h)\frac{(t-a)_{h}^{(-\nu)}y(a)}{\Gamma(1-\nu)}
-\frac{(t-a)_{h}^{(-\nu)}y^{2}(a)}{\Gamma(1-\nu)}
+\frac{(t-a)_{h}^{(-\nu)}}{\Gamma(1-\nu)}\big(y(t+\nu h)-y(a)\big)^{2}\\
&&~~~\ \ \ +\frac{h\nu}{\Gamma(1-\nu)}\sum_{s=\frac{a}{h}}^{\frac{t}{h}+\nu-1}
(t-\sigma(sh))_{h}^{(-\nu-1)}\big(y(t+\nu h)-y(sh)\big)^{2}\\
&&~~~=\frac{(t-a)_{h}^{(-\nu)}y^{2}(t+\nu h)}{\Gamma(1-\nu)}
+\frac{h\nu}{\Gamma(1-\nu)}\sum_{s=\frac{a}{h}}^{\frac{t}{h}+\nu-1}
(t-\sigma(sh))_{h}^{(-\nu-1)}\big(y(t+\nu h)-y(sh)\big)^{2}\\
&&~~~\geq 0.
\end{eqnarray*}
So, we have
$(_{a}\Delta_{h}^{\nu}y^{2})(t)\leq 2 y(t+\nu h)(_{a}\Delta_{h}^{\nu}y)(t)$.
The proof is complete.
\end{proof}

\begin{proposition}\label{l3.8}
For $\nu\in (0,1]$, and $y\in\R^{n}$,
the following relationship holds
\begin{equation}\label{3.18}
\frac{1}{2}{_{a}\Delta_{h}^{\nu}}(y^{T}(t)Py(t))\leq
y^{T}(t+\nu h)P (_{a}\Delta_{h}^{\nu}y)(t),\ \ t\in (h\N)_{a+(1-\nu)h},
\end{equation}
where $P\in \R^{n}\times \R^{n}$ is a constant, square,
symmetric, and positive definite matrix.
\end{proposition}

 \begin{proof}
 The proof is similar to Proposition \ref{l3.4},
 so we omit the details.
 \end{proof}

\begin{lemma}\label{l3.9}
Let $x=0$ be an equilibrium point of the system \eqref{3.2}.
Assume there exists a positive definite and
decrescent scalar function $V(t,x)$, class-$\mathcal{K}$ functions
$\gamma_{1}$, $\gamma_{2}$, and $\gamma_{3}$ such that
\begin{equation}\label{3.19}
\gamma_{1}(\|x(t)\|)\leq V(t,x(t))\leq \gamma_{2}(\|x(t)\|),\ \ t\in (h\N)_{a},
\end{equation}
and
\begin{equation}\label{3.20}
(_{a}\Delta_{h}^{\nu}V)(t,x(t))\leq -\gamma_{3}(\|x(t+\nu h)\|).
\end{equation}
Then the system \eqref{3.2} is asymptotically stable.
\end{lemma}

\begin{proof}
The proof is similar to Lemma \ref{l3.5},
so we omit the details.
\end{proof}

\begin{theorem}\label{t3.2}
Assume $x=0$ is an equilibrium point of the system \eqref{3.2}, if
the following condition is satisfied
\begin{equation*}
x^{T}(t+\nu h)P f(t,x(t+vh))\leq 0,\ \ t\in(h\N)_{a+(1-\nu)h},
\end{equation*}
then the system \eqref{3.2} is stable.
Also, if
\begin{equation*}
x^{T}(t+\nu h)P f(t,x(t+vh))< 0,\ \ t\in(h\N)_{a+(1-\nu)h}, \forall x\neq 0,
\end{equation*}
then the system \eqref{3.2} is asymptotically stable.
\end{theorem}

\begin{proof}
 Let us propose the following Lyapunov function,
 which is positive definite
 \begin{equation*}
 V(t,x(t))=\frac{1}{2}x^{T}(t)Px(t).
 \end{equation*}
Using Proposition \ref{l3.8}, we have
\begin{equation}\label{3.21}
(_{a}\Delta_{h}^{\nu}V)(t)\leq x^{T}(t+\nu h)P(_{a}\Delta_{h}^{\nu}x)(t)
=x^{T}(t+\nu h)Pf(t,x(t+\nu h))\leq 0,
\end{equation}
and by Lemma \ref{l3.6}, we have
\begin{equation*}
V(t,x(t))\leq V(a,x(a)),
\end{equation*}
that is,
\begin{equation*}
\frac{1}{2}x^{T}(t)Px(t)\leq \frac{1}{2}x^{T}(a)Px(a).
\end{equation*}
Similar to the proof of Theorem \ref{t3.1}, we have
\begin{equation*}
\lambda_{\min}\|x(t)\|^{2}\leq \lambda_{\max}\|x(a)\|^{2}.
\end{equation*}
According to the definition of stability in the sense of Lyapunov,
we obtain the system \eqref{3.2} is stable in the sense of Lyapunov.

If
\begin{equation*}
x^{T}(t+\nu h)P f(t,x(t+vh))< 0,\ \ t\in(h\N)_{a+(1-\nu)h}, \forall x\neq 0,
\end{equation*}
similar to the above step, we can show that
the system \eqref{3.2} is stable.
Using Proposition \ref{l3.8}, we have $(_{a}\Delta_{h}^{\nu}V)(t,x(t))\leq x^{T}(t+\nu h)P(_{a}\Delta_{h}^{\nu}x)(t)< 0$,
that is, the fractional order $h$-difference of
$V$ function is negative definite.
Given the relationship between positive definite functions and class-$\mathcal{K}$ functions in \cite{sl1991}.
Then, from Lemma \ref{l3.9},
the system \eqref{3.2} is asymptotically stable.
The proof is complete.
\end{proof}

\section{Polynomial Lyapunov functions for stability}

In this section, we will introduce several propositions,
which generalize the Lemma 2.10 in \cite{wb2017}
and Lemma 3.2 in \cite{bw2017},
which are very important to show the Lyapunov stability for the fractional order $h$-difference systems.

\begin{proposition}\label{l4.1}
For $\nu\in(0,1]$, $y\in \R$, $y(t)\geq0$, $t\in (h\N)_{a}$,
and $l\in \{2k+1,k\in \N_{1}\}$, the following inequality holds
\begin{equation}\label{4.1}
(_{a}\Delta_{h,\ast}^{\nu}y^{l})(t)\leq l y^{l-1}(t+\nu h)(_{a}\Delta_{h,\ast}^{\nu}y)(t),
\ \ t\in (h\N)_{a+(1-\nu)h}.
\end{equation}
\end{proposition}

\begin{proof}
We need to equivalently prove
\begin{equation}\label{4.2}
(_{a}\Delta_{h,\ast}^{\nu}y^{l})(t)- l y^{l-1}(t+\nu h)(_{a}\Delta_{h,\ast}^{\nu}y)(t)\leq0.
\end{equation}

For $\nu=1$, the inequality \eqref{4.2} can be written as
\begin{equation}\label{4.3}
(\Delta_{h} y^{l})(t)- l y^{l-1}(t+\nu h)(\Delta_{h} y)(t)\leq0.
\end{equation}
We will show the inequality \eqref{4.3} holds by induction.
When $l=3$, we have
\begin{equation*}
\begin{split}
3y^{2}(t+h)(\Delta_{h} y)(t)-(\Delta_{h} y^{3})(t)
&=3y^{2}(t+h)\frac{y(t+h)-y(t)}{h}-\frac{y^{3}(t+h)-y^{3}(t)}{h}\\
&=\frac{2y^{3}(t+h)-3y^{2}(t+h)y(t)+y^{3}(t)}{h}.
\end{split}
\end{equation*}
According to Proposition \ref{l3.3}, we have
\begin{equation*}
(\Delta_{h} y^{2})(t)- 2 y(t+h)(\Delta_{h} y)(t)\leq0,
\end{equation*}
that is,
\begin{equation}\label{4.4}
\frac{2y^{3}(t+h)-3y^{2}(t+h)y(t)}{h}\geq \frac{y^{2}(t+h)y(t)-2y^{2}(t)y(t+h)}{h}.
\end{equation}
Using the inequality \eqref{4.4}, and $y(t)\geq 0$, we have
\begin{equation*}
\begin{split}
3y^{2}(t+h)(\Delta_{h} y)(t)-(\Delta_{h} y^{3})(t)
\geq \frac{y^{3}(t)-2y^{2}(t)y(t+h)+y^{2}(t+h)y(t)}{h}
=\frac{y(t)\big(y(t+h)-y(t)\big)^{2}}{h}\geq 0.
\end{split}
\end{equation*}
So, we obtain
\begin{equation*}
(\Delta_{h} y^{3})(t)\leq 3 y^{2}(t+h)(\Delta_{h} y)(t).
\end{equation*}
Assume the inequality
\begin{equation*}
(\Delta_{h} y^{k})(t)- k y^{k-1}(t+ h)(\Delta_{h} y)(t)\leq0
\end{equation*}
 is true for $k=3$, $5$,$\cdot\cdot\cdot$, $l$,
 we will show the following case $k=l+2$ is true, that is,
\begin{equation*}
(\Delta_{h} y^{l+2})(t)\leq (l+2)y^{l+1}(t+h)(\Delta_{h} y)(t).
\end{equation*}
Then, we have
\begin{equation*}
\begin{split}
&(l+2)y^{l+1}(t+ h)(\Delta_{h} y)(t)-(\Delta_{h} y^{l+2})(t)\\
&~~~=(l+2)y^{l+1}(t+h)\frac{y(t+h)-y(t)}{h}
-\frac{y^{l+2}(t+h)-y^{l+2}(t)}{h}\\
&~~~=\frac{(l+1)y^{l+2}(t+h)-(l+2)y^{l+1}(t+ h)y(t)+y^{l+2}(t)}{h}.
\end{split}
\end{equation*}
From the induction assumption, we have
\begin{equation*}
\begin{split}
(\Delta_{h} y^{\frac{l+1}{2}+1})(t)\leq \Big(\frac{l+1}{2}+1\Big)y^{\frac{l+1}{2}}(t+h)(\Delta_{h} y)(t),
\end{split}
\end{equation*}
that is,
\begin{equation}\label{4.5}
\begin{split}
\frac{(l+1)y^{(l+2)}(t+h)-(l+2)y^{l+1}(t+h)y(t)}{h}
\geq \frac{y^{l+1}(t+h)y(t)
-2y^{\frac{l+1}{2}}(t+h)y^{\frac{l+1}{2}+1}(t)}{h}.
\end{split}
\end{equation}
Using the inequality \eqref{4.5}, and $y(t)\geq 0$, we have
\begin{equation*}
\begin{split}
(l+2)y^{l+1}(t+ h)(\Delta_{h} y)(t)-(\Delta_{h} y^{l+2})(t)
&\geq \frac{y^{l+1}(t+h)y(t)-2y^{\frac{l+1}{2}}(t+h)
y^{\frac{l+1}{2}+1}(t)+y^{l+2}(t)}{h}\\
&=\frac{y(t)\big(y^{\frac{l+1}{2}}(t+h)
-y^{\frac{l+1}{2}}(t)\big)^{2}}{h}\geq0.
\end{split}
\end{equation*}
So, we obtain
\begin{equation*}
(\Delta_{h} y^{l})(t)- l y^{l-1}(t+\nu h)(\Delta_{h} y)(t)\leq0,
\ \ l\in \{2k+1,k\in \N_{1}\}.
\end{equation*}

For $\nu\in(0,1)$, we prove the inequality \eqref{4.2} by induction.
First, we show the case $l=3$ holds. Using a summation by parts formula, we have
\begin{eqnarray*}
&&3y^{2}(t+\nu h)(_{a}\Delta_{h,\ast}^{\nu}y)(t)
-(_{a}\Delta_{h,\ast}^{\nu}y^{3})(t)\\
&&~~~=\frac{3y^{2}(t+\nu h)h}{\Gamma(1-\nu)}
\sum_{s=\frac{a}{h}}^{\frac{t}{h}+\nu-1}
(t-\sigma(sh))_{h}^{(-\nu)}(\Delta_{h} y)(sh)
-\frac{h}{\Gamma(1-\nu)}\sum_{s=\frac{a}{h}}^{\frac{t}{h}+\nu-1}
(t-\sigma(sh))_{h}^{(-\nu)}(\Delta_{h} y^{3})(sh)\\
&&~~~=\frac{3y^{2}(t+\nu h)h}{\Gamma(1-\nu)}
\Big[\frac{(t-sh)_{h}^{(-\nu)}y(sh)}{h}
\Big|_{s=\frac{a}{h}}^{\frac{t}{h}+\nu}
-\nu\sum_{s=\frac{a}{h}}^{\frac{t}{h}+\nu-1}
(t-\sigma(sh))_{h}^{(-\nu-1)}y(sh)\Big]\\
&&~~~\ \ \ -\frac{h}{\Gamma(1-\nu)}\Big[\frac{(t-sh)_{h}^{(-\nu)}
y^{3}(sh)}{h}\Big|_{s=\frac{a}{h}}^{\frac{t}{h}+\nu}
-\nu\sum_{s=\frac{a}{h}}^{\frac{t}{h}+\nu-1}
(t-\sigma(sh))_{h}^{(-\nu-1)}y^{3}(sh)\Big]\\
&&~~~=2h^{-\nu}y^{3}(t+\nu h)
-\frac{3(t-a)_{h}^{(-\nu)}y^{2}(t+\nu h)y(a)}{\Gamma(1-\nu)}
-\frac{3y^{2}(t+\nu h)h\nu}{\Gamma(1-\nu)}
\sum_{s=\frac{a}{h}}^{\frac{t}{h}+\nu-1}
(t-\sigma(sh))_{h}^{(-\nu-1)}y(sh)\\
&&~~~ \ \ \ +\frac{(t-a)_{h}^{(-\nu)}y^{3}(a)}{\Gamma(1-\nu)}
+\frac{h\nu}{\Gamma(1-\nu)}
\sum_{s=\frac{a}{h}}^{\frac{t}{h}+\nu-1}
(t-\sigma(sh))_{h}^{(-\nu-1)}y^{3}(sh).
\end{eqnarray*}
It follows from Proposition \ref{l3.3} that
$$
(_{a}\Delta_{h,\ast}^{\nu}y^{2})(t)\leq2 y(t+\nu h)(_{a}\Delta_{h,\ast}^{\nu}y)(t).
$$
Applying a summation by parts formula, we have
\begin{eqnarray*}
&&2 y(t+\nu h)(_{a}\Delta_{h,\ast}^{\nu}y)(t)- (_{a}\Delta_{h,\ast}^{\nu}y^{2})(t)\\
&&~~~=\frac{2 y(t+\nu h)h}{\Gamma(1-\nu)}
\sum_{s=\frac{a}{h}}^{\frac{t}{h}+\nu-1}
(t-\sigma(sh))_{h}^{(-\nu)}(\Delta_{h} y)(sh)
-\frac{h}{\Gamma(1-\nu)}\sum_{s=\frac{a}{h}}^{\frac{t}{h}+\nu-1}
(t-\sigma(sh))_{h}^{(-\nu)}(\Delta_{h} y^{2})(sh)\\
&&~~~=\frac{2 y(t+\nu h)h}{\Gamma(1-\nu)}\Big[\frac{(t-sh)_{h}^{(-\nu)}
y(sh)}{h}\Big|_{s=\frac{a}{h}}^{\frac{t}{h}+\nu}
-\nu\sum_{s=\frac{a}{h}}^{\frac{t}{h}+\nu-1}
(t-\sigma(sh))_{h}^{(-\nu-1)}y(sh)\Big]\\
&&~~~ \ \ \ -\frac{h}{\Gamma(1-\nu)}\Big[\frac{(t-sh)_{h}^{(-\nu)}
y^{2}(sh)}{h}\Big|_{s=\frac{a}{h}}^{\frac{t}{h}+\nu}
-\nu\sum_{s=\frac{a}{h}}^{\frac{t}{h}+\nu-1}
(t-\sigma(sh))_{h}^{(-\nu-1)}y^{2}(sh)\Big]\\
&&~~~=h^{-\nu}y^{2}(t+\nu h)-\frac{2(t-a)_{h}^{(-\nu)}
y(t+\nu h)y(a)}{\Gamma(1-\nu)}
-\frac{2y(t+\nu h)h\nu}{\Gamma(1-\nu)}
\sum_{s=\frac{a}{h}}^{\frac{t}{h}+\nu-1}
(t-\sigma(sh))_{h}^{(-\nu-1)}y(sh)\\
&&~~~\ \ \ +\frac{(t-a)_{h}^{(-\nu)}y^{2}(a)}{\Gamma(1-\nu)}
+\frac{h\nu}{\Gamma(1-\nu)}
\sum_{s=\frac{a}{h}}^{\frac{t}{h}+\nu-1}
(t-\sigma(sh))_{h}^{(-\nu-1)}y^{2}(sh)
\geq0,
\end{eqnarray*}
that is,
\begin{equation}\label{4.6}
\begin{split}
&2h^{-\nu}y^{3}(t+\nu h)\\
&~~~\geq \frac{4(t-a)_{h}^{(-\nu)}y^{2}(t+\nu h)y(a)}{\Gamma(1-\nu)}
+\frac{4y^{2}(t+\nu h)h\nu}{\Gamma(1-\nu)}
\sum_{s=\frac{a}{h}}^{\frac{t}{h}+\nu-1}
(t-\sigma(sh))_{h}^{(-\nu-1)}y(sh)\\
&~~~\ \ \ -\frac{2(t-a)_{h}^{(-\nu)}y(t+\nu h)y^{2}(a)}{\Gamma(1-\nu)}
-\frac{2y(t+\nu h)h\nu}{\Gamma(1-\nu)}
\sum_{s=\frac{a}{h}}^{\frac{t}{h}+\nu-1}
(t-\sigma(sh))_{h}^{(-\nu-1)}y^{2}(sh).
\end{split}
\end{equation}
Using the inequality \eqref{4.6}, and $y(t)\geq0$, we have
\begin{eqnarray*}
&&3y^{2}(t+\nu h)(_{a}\Delta_{h,\ast}^{\nu}y)(t)- (_{a}\Delta_{h,\ast}^{\nu}y^{3})(t)\\
&&~~~\geq \frac{(t-a)_{h}^{(-\nu)}y^{2}(t+\nu h)y(a)}{\Gamma(1-\nu)}+\frac{y^{2}(t+\nu h)h\nu}{\Gamma(1-\nu)}\sum_{s=\frac{a}{h}}^{\frac{t}{h}+\nu-1}
(t-\sigma(sh))_{h}^{(-\nu-1)}y(sh)\\
&&~~~\ \ \ -\frac{2(t-a)_{h}^{(-\nu)}y(t+\nu h)y^{2}(a)}{\Gamma(1-\nu)}-\frac{2 y(t+\nu h)h\nu}{\Gamma(1-\nu)}\sum_{s=\frac{a}{h}}^{\frac{t}{h}+\nu-1}
(t-\sigma(sh))_{h}^{(-\nu-1)}y^{2}(sh)\\
&&~~~\ \ \ +\frac{(t-a)_{h}^{(-\nu)}y^{3}(a)}{\Gamma(1-\nu)}
+\frac{h\nu}{\Gamma(1-\nu)}
\sum_{s=\frac{a}{h}}^{\frac{t}{h}+\nu-1}
(t-\sigma(sh))_{h}^{(-\nu-1)}y^{3}(sh)\\
&&~~~=\frac{(t-a)_{h}^{(-\nu)}y(a)}{\Gamma(1-\nu)}
\big(y(t+\nu h)-y(a)\big)^{2}
+\frac{h\nu}{\Gamma(1-\nu)}
\sum_{s=\frac{a}{h}}^{\frac{t}{h}+\nu-1}
(t-\sigma(sh))_{h}^{(-\nu-1)}
y(sh)\big(y(t+\nu h)-y(sh)\big)^{2}\geq0.
\end{eqnarray*}
So, we obtain
\begin{equation*}
(_{a}\Delta_{h,\ast}^{\nu}y^{3})(t)-3 y^{2}(t+\nu h)(_{a}\Delta_{h,\ast}^{\nu}y)(t)\leq0.
\end{equation*}
Now, we assume the inequality
\begin{equation*}
(_{a}\Delta_{h,\ast}^{\nu}y^{k})(t)\leq ky^{k-1}(t+\nu h)(_{a}\Delta_{h,\ast}^{\nu}y)(t)
\end{equation*}
is true for $k=3$, $5$, $\cdot\cdot\cdot$, $l$,
we will show the following $k=l+2$ case is true, that is,
\begin{equation*}
(_{a}\Delta_{h,\ast}^{\nu}y^{l+2})(t)\leq (l+2)y^{l+1}(t+\nu h)(_{a}\Delta_{h,\ast}^{\nu}y)(t).
\end{equation*}
Using summation by parts formula, we have
\begin{eqnarray*}
&&(l+2)y^{l+1}(t+\nu h)(_{a}\Delta_{h,\ast}^{\nu}y)(t)
-(_{a}\Delta_{h,\ast}^{\nu}y^{l+2})(t)\\
&&~~~=\frac{(l+2)y^{l+1}(t+\nu h)h}{\Gamma(1-\nu)}
\sum_{s=\frac{a}{h}}^{\frac{t}{h}+\nu-1}
(t-\sigma(sh))_{h}^{(-\nu)}(\Delta_{h} y)(sh)
-\frac{h}{\Gamma(1-\nu)}
\sum_{s=\frac{a}{h}}^{\frac{t}{h}+\nu-1}
(t-\sigma(sh))_{h}^{(-\nu)}(\Delta_{h} y^{l+2})(sh)\\
&&~~~=\frac{(l+2)y^{l+1}(t+\nu h)h}{\Gamma(1-\nu)}
\Big[\frac{(t-sh)_{h}^{(-\nu)}y(sh)}{h}
\Big|_{s=\frac{a}{h}}^{\frac{t}{h}+\nu}
-\nu \sum_{s=\frac{a}{h}}^{\frac{t}{h}+\nu-1}
(t-\sigma(sh))_{h}^{(-\nu-1)}y(sh) \Big]\\
&&~~~\ \ \ -\frac{h}{\Gamma(1-\nu)}\Big[ \frac{(t-sh)_{h}^{(-\nu)}y^{l+2}(sh)}{h}
\Big|_{s=\frac{a}{h}}^{\frac{t}{h}+\nu}
-\nu \sum_{s=\frac{a}{h}}^{\frac{t}{h}+\nu-1}
(t-\sigma(sh))_{h}^{(-\nu-1)}y^{l+2}(sh) \Big]\\
&&~~~=(l+1)h^{-\nu}y^{l+2}(t+\nu h)-\frac{(l+2)(t-a)_{h}^{(-\nu)}y^{l+1}
(t+\nu h)y(a)}{\Gamma(1-\nu)}
+\frac{(t-a)_{h}^{(-\nu)}y^{l+2}(a)}{\Gamma(1-\nu)}\\
&&~~~\ \ \ -\frac{(l+2)y^{l+1}(t+\nu h)h\nu}{\Gamma(1-\nu)}
\sum_{s=\frac{a}{h}}^{\frac{t}{h}+\nu-1}(t-\sigma(sh))_{h}^{(-\nu-1)}y(sh)
+\frac{h\nu}{\Gamma(1-\nu)}\sum_{s=\frac{a}{h}}^{\frac{t}{h}+\nu-1}
(t-\sigma(sh))_{h}^{(-\nu-1)}y^{l+2}(sh).
\end{eqnarray*}
 From the induction assumption, we have
\begin{equation*}
(_{a}\Delta_{h,\ast}^{\nu}y^{\frac{l+1}{2}+1})(t)
\leq \Big(\frac{l+1}{2}+1\Big)y^{\frac{l+1}{2}}(t+\nu h)(_{a}\Delta_{h,\ast}^{\nu}y)(t).
\end{equation*}
Using a summation by parts formula, we get
\begin{eqnarray*}
&&\Big(\frac{l+1}{2}+1\Big)y^{\frac{l+1}{2}}(t+\nu h)(_{a}\Delta_{h,\ast}^{\nu}y)(t)
- (_{a}\Delta_{h,\ast}^{\nu}y^{\frac{l+1}{2}+1})(t)\\
&&~~~=\frac{(\frac{l+1}{2}+1)y^{\frac{l+1}{2}}(t+\nu h)h}{\Gamma(1-\nu)}
\sum_{s=\frac{a}{h}}^{\frac{t}{h}+\nu-1}
(t-\sigma(sh))_{h}^{(-\nu)}(\Delta_{h} y)(sh) -\frac{h}{\Gamma(1-\nu)}\sum_{s=\frac{a}{h}}^{\frac{t}{h}+\nu-1}
(t-\sigma(sh))_{h}^{(-\nu)}(\Delta_{h} y^{\frac{l+1}{2}+1})(sh)\\
&&~~~=\frac{(\frac{l+1}{2}+1)y^{\frac{l+1}{2}}(t+\nu h)h}{\Gamma(1-\nu)}
\Big[\frac{(t-sh)_{h}^{(-\nu)}y(sh)}{h}
\Big|_{s=\frac{a}{h}}^{\frac{t}{h}+\nu}
-\nu \sum_{s=\frac{a}{h}}^{\frac{t}{h}+\nu-1}
(t-\sigma(sh))_{h}^{(-\nu-1)}y(sh) \Big]\\
&&~~~\ \ \ -\frac{h}{\Gamma(1-\nu)}\Big[\frac{(t-sh)_{h}^{(-\nu)}
y^{\frac{l+1}{2}+1}(sh)}{h}\Big|_{s=\frac{a}{h}}^{\frac{t}{h}+\nu}
-\nu \sum_{s=\frac{a}{h}}^{\frac{t}{h}+\nu-1}
(t-\sigma(sh))_{h}^{(-\nu-1)}y^{\frac{l+1}{2}+1}(sh) \Big]\\
&&~~~=\frac{l+1}{2}h^{-\nu}y^{\frac{l+1}{2}+1}(t+\nu h)
-\frac{(\frac{l+1}{2}+1)(t-a)_{h}^{(-\nu)}y^{\frac{l+1}{2}}(t+\nu h)y(a)}{\Gamma(1-\nu)}\\
&&~~~\ \ \ -\frac{(\frac{l+1}{2}+1)y^{\frac{l+1}{2}}(t+\nu h)h\nu}{\Gamma(1-\nu)}\sum_{s=\frac{a}{h}}^{\frac{t}{h}+\nu-1}
(t-\sigma(sh))_{h}^{(-\nu-1)}y(sh)+
\frac{(t-a)_{h}^{(-\nu)}y^{\frac{l+1}{2}+1}(a)}{\Gamma(1-\nu)}\\
&&~~~\ \ \ +\frac{h\nu}{\Gamma(1-\nu)}
\sum_{s=\frac{a}{h}}^{\frac{t}{h}+\nu-1}
(t-\sigma(sh))_{h}^{(-\nu-1)}y^{\frac{l+1}{2}+1}(sh)\geq0,
\end{eqnarray*}
that is,
\begin{equation}\label{4.7}
\begin{split}
&(l+1) h^{-\nu}y^{l+2}(t+\nu h)\\
&\ \ \ \geq\frac{(l+3)(t-a)_{h}^{(-\nu)}y^{l+1}(t+\nu h)y(a)}{\Gamma(1-\nu)}
+\frac{(l+3)y^{l+1}(t+\nu h)h\nu}{\Gamma(1-\nu)}
\sum_{s=\frac{a}{h}}^{\frac{t}{h}+\nu-1}
(t-\sigma(sh))_{h}^{(-\nu-1)}y(sh)\\
&\ \ \ \ \ \ -\frac{2(t-a)_{h}^{(-\nu)}y^{\frac{l+1}{2}}(t+\nu h)y^{\frac{l+1}{2}+1}(a)}{\Gamma(1-\nu)}
-\frac{2y^{\frac{l+1}{2}}(t+\nu h)h\nu}{\Gamma(1-\nu)}
\sum_{s=\frac{a}{h}}^{\frac{t}{h}+\nu-1}
(t-\sigma(sh))_{h}^{(-\nu-1)}y^{\frac{l+1}{2}+1}(sh).
\end{split}
\end{equation}
Using the inequality \eqref{4.7}, and $y(t)\geq0$, we obtain
\begin{eqnarray*}
&&(l+2)y^{l+1}(t+\nu h)(_{a}\Delta_{h,\ast}^{\nu}y)(t)
-(_{a}\Delta_{h,\ast}^{\nu}y^{l+2})(t)\\
&&~~~\geq \frac{(t-a)_{h}^{(-\nu)}y^{l+1}(t+\nu h)y(a)}{\Gamma(1-\nu)}
+\frac{y^{l+1}(t+\nu h)h\nu}{\Gamma(1-\nu)}
\sum_{s=\frac{a}{h}}^{\frac{t}{h}+\nu-1}(t-\sigma(sh))_{h}^{(-\nu-1)}y(sh)\\
&&~~~\ \ \ -\frac{2(t-a)_{h}^{(-\nu)}y^{\frac{l+1}{2}}(t+\nu h)
y^{\frac{l+1}{2}+1}(a)}{\Gamma(1-\nu)}
-\frac{2y^{\frac{l+1}{2}}(t+\nu h)h\nu}{\Gamma(1-\nu)}
\sum_{s=\frac{a}{h}}^{\frac{t}{h}+\nu-1}
(t-\sigma(sh))_{h}^{(-\nu-1)}y^{\frac{l+1}{2}+1}(sh)\\
&&~~~\ \ \ +\frac{(t-a)_{h}^{(-\nu)}y^{l+2}(a)}{\Gamma(1-\nu)}
+\frac{h\nu}{\Gamma(1-\nu)}\sum_{s=\frac{a}{h}}^{\frac{t}{h}+\nu-1}
(t-\sigma(sh))_{h}^{(-\nu-1)}y^{l+2}(sh)\\
&&~~~=\frac{(t-a)_{h}^{(-\nu)}y(a)}{\Gamma(1-\nu)}
\big(y^{\frac{l+1}{2}}(t+\nu h)-y^{\frac{l+1}{2}}(a)\big)^{2}\\
&&~~~\ \ \ +\frac{h\nu}{\Gamma(1-\nu)}
\sum_{s=\frac{a}{h}}^{\frac{t}{h}+\nu-1}(t-\sigma(sh))_{h}^{(-\nu-1)}
y(sh)\big(y^{\frac{l+1}{2}}(t+\nu h)
-y^{\frac{l+1}{2}}(sh)\big)^{2}\geq0.
\end{eqnarray*}
So, we obtain
\begin{equation*}
(_{a}\Delta_{h,\ast}^{\nu}y^{l})(t)-l y^{l-1}(t+\nu h)(_{a}\Delta_{h,\ast}^{\nu}y)(t)\leq0,
\ \ l\in \{2k+1,k\in \N_{1}\}.
\end{equation*}
The proof is complete.
\end{proof}

\begin{proposition}\label{l4.2}
For $\nu\in(0,1]$, $y\in\R$,
and $m\in \N_{1}$, the following inequality holds
\begin{equation}\label{4.8}
(_{a}\Delta_{h,\ast}^{\nu}y^{2^{m}})(t)
\leq 2^{m}y^{(2^{m}-1)}(t+\nu h)
(_{a}\Delta_{h,\ast}^{\nu}y)(t),\ \ t\in (h\N)_{a+(1-\nu)h}.
\end{equation}
\end{proposition}

\begin{proof}
To prove the inequality \eqref{4.8},
we start by iterating $m$ times, and using Proposition \ref{l3.3},
which results in
\begin{equation*}
\begin{split}
(_{a}\Delta_{h,\ast}^{\nu}y^{2^{m}})(t)&\leq 2y^{2^{m-1}}(t+\nu h) (_{a}\Delta_{h,\ast}^{\nu}y^{2^{m-1}})(t)\\
&\leq 2^{2}y^{2^{m-1}}(t+\nu h)y^{2^{m-2}}(t+\nu h) (_{a}\Delta_{h,\ast}^{\nu}y^{2^{m-2}})(t)\leq\cdot\cdot\cdot\\
&\leq 2^{m}y^{2^{m-1}}(t+\nu h)y^{2^{m-2}}(t+\nu h)\cdot\cdot\cdot y^{2^{0}}(t+\nu h)(_{a}\Delta_{h,\ast}^{\nu}y^{2^{0}})(t),
\end{split}
\end{equation*}
which is equivalent to
\begin{equation*}
(_{a}\Delta_{h,\ast}^{\nu}y^{2^{m}})(t)
\leq 2^{m}y^{(2^{m}-1)}(t+\nu h)(_{a}\Delta_{h,\ast}^{\nu}y)(t).
\end{equation*}
The proof is complete.
\end{proof}

\begin{theorem}\label{t4.1}
Assume $x=0$ is an equilibrium point of the system \eqref{3.1}.
Then the following statements are satisfied:

(i) For $x_{i}(t)\geq 0$ $(i=1,2,\cdots,n)$,
$t\in (h\N)_{a}$, and $l\in\{2k+1, k\in\N_{1}\}$, if
the following condition is satisfied
\begin{equation*}
x_{i}^{l-1}(t+\nu h)f_{i}(t,x(t+vh))\leq 0,\ \ t\in(h\N)_{a+(1-\nu)h},
\end{equation*}
then the system \eqref{3.1} is stable.
Also, if
\begin{equation*}
x_{i}^{l-1}(t+\nu h)f_{i}(t,x(t+vh))< 0,\ \ t\in(h\N)_{a+(1-\nu)h}, \forall x_{i}\neq 0,
\end{equation*}
then the system \eqref{3.1} is asymptotically stable.

(ii) For $m\in\N_{1}$, if the following condition is satisfied
\begin{equation*}
x_{i}^{(2^{m}-1)}(t+\nu h)f_{i}(t,x(t+vh))\leq 0,\ \ t\in(h\N)_{a+(1-\nu)h},
\end{equation*}
then the system \eqref{3.1} is stable.
Also, if
\begin{equation*}
x_{i}^{(2^{m}-1)}(t+\nu h)f_{i}(t,x(t+vh))< 0,\ \ t\in(h\N)_{a+(1-\nu)h}, \forall x_{i}\neq 0,
\end{equation*}
then the system \eqref{3.1} is asymptotically stable.
\end{theorem}

\begin{proof}
(i) Let us propose the following Lyapunov function,
 which is positive definite
\begin{equation*}
 V(t,x(t))=\sum_{i=1}^{n}\frac{x_{i}^{l}(t)}{l}.
\end{equation*}
Using Proposition \ref{l4.1}, we obtain
\begin{equation*}
(_{a}\Delta_{h,\ast}^{\nu}V)(t)\leq
\sum_{i=1}^{n}x_{i}^{l-1}(t+\nu h)(_{a}\Delta_{h,\ast}^{\nu}x_{i})(t)
=\sum_{i=1}^{n}x_{i}^{l-1}(t+\nu h)f_{i}(t,x(t+\nu h))\leq 0.
\end{equation*}
Hence, by Lemma \ref{l3.2}, we have
\begin{equation*}
V(t,x(t))\leq V(a,x(a)),
\end{equation*}
that is,
\begin{equation*}
\sum_{i=1}^{n}\frac{x_{i}^{l}(t)}{l}\leq
\sum_{i=1}^{n}\frac{x_{i}^{l}(a)}{l}.
\end{equation*}
According to the definition of stability in the sense of Lyapunov,
we obtain the system \eqref{3.1} is stable in the sense of Lyapunov.

If
\begin{equation*}
x_{i}^{l-1}(t+\nu h) f_{i}(t,x(t+vh))< 0,\ \ t\in(h\N)_{a+(1-\nu)h}, \forall x\neq 0,
\end{equation*}
similar to the above step, we can show that
the system \eqref{3.1} is stable.
Using Proposition \ref{l4.1}, we have $(_{a}\Delta_{h,\ast}^{\nu}V)(t,x(t))\leq \sum_{i=1}^{n}x_{i}^{l-1}(t+\nu h)
(_{a}\Delta_{h,\ast}^{\nu}x_{i})(t)<0$,
that is, the fractional order $h$-difference of
$V$ function is negative definite.
Given the relationship between positive definite functions and class-$\mathcal{K}$ functions in \cite{sl1991}.
It follows from Lemma \ref{l3.5} that the system \eqref{3.1} is asymptotically stable.

(ii) The proof is similar to the previous one, by Proposition \ref{l4.2}
and the positive definite Lyapunov function:
\begin{equation*}
 V(t,x(t))=\sum_{i=1}^{n}\frac{x_{i}^{2^{m}}(t)}{2^{m}}.
\end{equation*}
The proof is complete.
\end{proof}

\begin{proposition}\label{l4.3}
For $\nu\in(0,1]$, $y\in\R$, $y(t)\geq0$, $t\in (h\N)_{a}$,
and $l\in\{2k+1,k\in\N_{1}\}$, the following inequality holds
\begin{equation}\label{4.9}
(_{a}\Delta_{h}^{\nu}y^{l})(t)\leq l y^{l-1}(t+\nu h)(_{a}\Delta_{h}^{\nu}y)(t),\ \ t\in (h\N)_{a+(1-\nu)h}.
\end{equation}
\end{proposition}

\begin{proof}
We need to equivalently prove
\begin{equation}\label{4.10}
(_{a}\Delta_{h}^{\nu}y^{l})(t)- l y^{l-1}(t+\nu h)(_{a}\Delta_{h}^{\nu}y)(t)\leq0.
\end{equation}

For $\nu=1$, we can show as in the proof of Proposition \ref{l4.1}.
For $\nu\in(0,1)$, we show the inequality \eqref{4.10} by induction.
When $l=3$, using Lemma \ref{l2.1}, we have
\begin{eqnarray*}
&&3 y^{2}(t+\nu h)(_{a}\Delta_{h}^{\nu}y)(t)
-(_{a}\Delta_{h}^{\nu}y^{3})(t)\\
&&~~~=\frac{h}{\Gamma(-\nu)}\sum_{s=\frac{a}{h}}^{\frac{t}{h}+\nu}
(t-\sigma(sh))_{h}^{(-\nu-1)}\big(3y^{2}(t+\nu h)y(sh)-y^{3}(sh)\big)\\
&&~~~=2h^{-\nu}y^{3}(t+\nu h)+\frac{h}{\Gamma(-\nu)}
\sum_{s=\frac{a}{h}}^{\frac{t}{h}+\nu-1}
(t-\sigma(sh))_{h}^{(-\nu-1)}\big(3y^{2}(t+\nu h)y(sh)-y^{3}(sh)\big).
\end{eqnarray*}
It follows from Proposition \ref{l3.7} that
\begin{equation*}
(_{a}\Delta_{h}^{\nu}y^{2})(t)\leq2 y(t+\nu h) (_{a}\Delta_{h}^{\nu}y)(t).
\end{equation*}
Using Lemma \ref{l2.1}, we have
\begin{equation*}
\begin{split}
&2 y(t+\nu h)(_{a}\Delta_{h}^{\nu}y)(t)- (_{a}\Delta_{h}^{\nu}y^{2})(t)\\
&~~~=\frac{h}{\Gamma(-\nu)}\sum_{s=\frac{a}{h}}^{\frac{t}{h}+\nu}
(t-\sigma(sh))_{h}^{(-\nu-1)}\big(2y(t+\nu h)y(sh)-y^{2}(sh)\big)\\
&~~~=h^{-\nu}y^{2}(t+\nu h)+\frac{h}{\Gamma(-\nu)}
\sum_{s=\frac{a}{h}}^{\frac{t}{h}+\nu-1}(t-\sigma(sh))_{h}^{(-\nu-1)}
\big(2y(t+\nu h)y(sh)-y^{2}(sh)\big)\geq0,
\end{split}
\end{equation*}
that is,
\begin{equation}\label{4.11}
2h^{-\nu}y^{3}(t+\nu h)\geq \frac{h}{\Gamma(-\nu)}
\sum_{s=\frac{a}{h}}^{\frac{t}{h}+\nu-1}(t-\sigma(sh))_{h}^{(-\nu-1)}
\big(2y(t+\nu h)y^{2}(sh)-4y^{2}(t+\nu h)y(sh)\big).
\end{equation}
Using the inequality \eqref{4.11}, and $y(t)\geq0$, we have
\begin{eqnarray*}
&&3 y^{2}(t+\nu h)(_{a}\Delta_{h}^{\nu}y)(t)
-(_{a}\Delta_{h}^{\nu}y^{3})(t)\\
&&~~~\geq \frac{h}{\Gamma(-\nu)}\sum_{s=\frac{a}{h}}^{\frac{t}{h}+\nu-1}
(t-\sigma(sh))_{h}^{(-\nu-1)}
\big(2y(t+\nu h)y^{2}(sh)-4y^{2}(t+\nu h)y(sh)\big)\\
&&~~~\ \ \ +\frac{h}{\Gamma(-\nu)}\sum_{s=\frac{a}{h}}^{\frac{t}{h}+\nu-1}
(t-\sigma(sh))_{h}^{(-\nu-1)}
\big(3y^{2}(t+\nu h)y(sh)-y^{3}(sh)\big)\\
&&~~~=-\frac{h}{\Gamma(-\nu)}\sum_{s=\frac{a}{h}}^{\frac{t}{h}+\nu-1}
(t-\sigma(sh))_{h}^{(-\nu-1)}
y(sh)\big(y^{2}(t+\nu h)+y^{2}(sh)-2y(t+\nu h)y(sh)\big)\\
&&~~~=-\frac{h}{\Gamma(-\nu)}\sum_{s=\frac{a}{h}}^{\frac{t}{h}+\nu-1}
(t-\sigma(sh))_{h}^{(-\nu-1)}
y(sh)\big(y(t+\nu h)+y(sh)\big)^{2}\geq0.
\end{eqnarray*}
So, we obtain
\begin{equation*}
(_{a}\Delta_{h}^{\nu}y^{3})(t)-3 y^{2}
(t+\nu h)(_{a}\Delta_{h}^{\nu}y)(t)\leq0.
\end{equation*}
Now, we assume the inequality
\begin{equation*}
(_{a}\Delta_{h}^{\nu}y^{k})(t)\leq ky^{k-1}(t+\nu h) (_{a}\Delta_{h}^{\nu}y)(t)
\end{equation*}
is true for $k=3$, $5$, $\cdot\cdot\cdot$, $l$,
we will show the following $k=l+2$ case is true, that is,
\begin{equation*}
(_{a}\Delta_{h}^{\nu}y^{l+2})(t)\leq (l+2)y^{l+1}(t+\nu h) (_{a}\Delta_{h}^{\nu}y)(t).
\end{equation*}
Using Lemma \ref{l2.1}, we have
\begin{equation*}
\begin{split}
&(l+2)y^{l+1}(t+\nu h)(_{a}\Delta_{h}^{\nu}y)(t)
-(_{a}\Delta_{h}^{\nu}y^{l+2})(t)\\
&~~~=\frac{(l+2)y^{l+1}(t+\nu h)h} {\Gamma(-\nu)}\sum_{s=\frac{a}{h}}^{\frac{t}{h}+\nu}
(t-\sigma(sh))_{h}^{(-\nu-1)}y(sh)
-\frac{h}{\Gamma(-\nu)}\sum_{s=\frac{a}{h}}^{\frac{t}{h}+\nu}
(t-\sigma(sh))_{h}^{(-\nu-1)}y^{l+2}(sh)\\
&~~~=(l+1) h^{-\nu}y^{l+2}(t+\nu h)
+\frac{h}{\Gamma(-\nu)}
\sum_{s=\frac{a}{h}}^{\frac{t}{h}+\nu-1}
(t-\sigma(sh))_{h}^{(-\nu-1)}
\big[(l+2)y^{l+1}(t+\nu h)y(sh)-y^{l+2}(sh)\big].
\end{split}
\end{equation*}
From the induction assumption, we have
\begin{equation*}
(_{a}\Delta_{h}^{\nu}y^{\frac{l+1}{2}+1})(t)
\leq \Big(\frac{l+1}{2}+1\Big)y^{\frac{l+1}{2}}(t+\nu h) (_{a}\Delta_{h}^{\nu}y)(t).
\end{equation*}
By Lemma \ref{l2.1}, we have
\begin{eqnarray*}
&&\Big(\frac{l+1}{2}+1\Big)y^{\frac{l+1}{2}}(t+\nu h) (_{a}\Delta_{h}^{\nu}y)(t)
- (_{a}\Delta_{h}^{\nu}y^{\frac{l+1}{2}+1})(t)\\
&&~~~=\frac{(\frac{l+1}{2}+1)y^{\frac{l+1}{2}}
(t+\nu h)h}{\Gamma(-\nu)}
\sum_{s=\frac{a}{h}}^{\frac{t}{h}+\nu}
(t-\sigma(sh))_{h}^{(-\nu-1)}y(sh)
-\frac{h}{\Gamma(-\nu)}\sum_{s=\frac{a}{h}}^{\frac{t}{h}+\nu}
(t-\sigma(sh))_{h}^{(-\nu-1)}y^{\frac{l+1}{2}+1}(sh)\\
&&~~~=\frac{l+1}{2} h^{-\nu}y^{\frac{l+1}{2}+1}(t+\nu h)
+\frac{h}{\Gamma(-\nu)}\sum_{s=\frac{a}{h}}^{\frac{t}{h}+\nu-1}
(t-\sigma(sh))_{h}^{(-\nu-1)}
\Big[(\frac{l+1}{2}+1)y^{\frac{l+1}{2}}
(t+\nu h)y(sh)-y^{\frac{l+1}{2}+1}(sh)\Big]\geq0,
\end{eqnarray*}
that is,
\begin{equation}\label{4.12}
\begin{split}
&l h^{-\nu}y^{l+2}(t+\nu h)
\geq\frac{h}{\Gamma(-\nu)}\sum_{s=\frac{a}{h}}^{\frac{t}{h}+\nu-1}
(t-\sigma(sh))_{h}^{(-\nu-1)}
\big[-(l+3)y^{l+1}(t+\nu h)y(sh)+2y^{\frac{l+1}{2}}
(t+\nu h)y^{\frac{l+1}{2}+1}(sh)\big].
\end{split}
\end{equation}
Using the inequality \eqref{4.12}, and $y(t)\geq0$, we have
\begin{eqnarray*}
&&(l+2)y^{l+1}(t+\nu h)(_{a}\Delta_{h}^{\nu}y)(t)
-(_{a}\Delta_{h}^{\nu}y^{l+2})(t)\\
&&~~~\geq \frac{h}{\Gamma(-\nu)}\sum_{s=\frac{a}{h}}^{\frac{t}{h}+\nu-1}
(t-\sigma(sh))_{h}^{(-\nu-1)}\big[-(l+3)y^{l+1}(t+\nu h)y(sh)
+2y^{\frac{l+1}{2}}(t+\nu h)y^{\frac{l+1}{2}+1}(sh)\big]\\
&&~~~\ \ \ +\frac{h}{\Gamma(-\nu)}\sum_{s=\frac{a}{h}}^{\frac{t}{h}+\nu-1}
(t-\sigma(sh))_{h}^{(-\nu-1)}
\big[(l+2)y^{l+1}(t+\nu h)y(sh)-y^{l+2}(sh)\big]\\
&&~~~=\frac{-h}{\Gamma(-\nu)}\sum_{s=\frac{a}{h}}^{\frac{t}{h}+\nu-1}
(t-\sigma(sh))_{h}^{(-\nu-1)}y(sh)\big(y^{l+1}(t+\nu h)+y^{l+1}(sh)
-2y^{\frac{l+1}{2}}(t+\nu h)y^{\frac{l+1}{2}}(sh)\big)\\
&&~~~=\frac{-h}{\Gamma(-\nu)}\sum_{s=\frac{a}{h}}^{\frac{t}{h}+\nu-1}
(t-\sigma(sh))_{h}^{(-\nu-1)}y(sh)\big(y^{\frac{l+1}{2}}(t+\nu h) -y^{\frac{l+1}{2}}(sh)\big)^{2}\geq0,
\end{eqnarray*}
So, we obtain
\begin{equation*}
(_{a}\Delta_{h}^{\nu}y^{l})(t)-l y^{l-1}(t+\nu h) (_{a}\Delta_{h}^{\nu}y)(t)\leq0,\ \ l\in\{2k+1,k\in\N_{1}\}.
\end{equation*}
The proof is complete.
\end{proof}

\begin{proposition}\label{l4.4}
For $\nu\in(0,1]$, $y\in\R$,
and $m\in\N_{1}$, the following inequality holds
\begin{equation}\label{4.13}
(_{a}\Delta_{h}^{\nu}y^{2^{m}})(t)\leq 2^{m}y^{(2^{m}-1)}(t+\nu h) (_{a}\Delta_{h}^{\nu}y)(t),\ \ t\in (h\N)_{a+(1-\nu)h}.
\end{equation}
\end{proposition}

\begin{proof}
To prove the inequality \eqref{4.13},
we start by iterating $m$ times, and using Proposition \ref{l3.7},
which results in
\begin{equation*}
\begin{split}
(_{a}\Delta_{h}^{\nu}y^{2^{m}})(t)&\leq 2y^{2^{m-1}}(t+\nu h) (_{a}\Delta_{h}^{\nu}y^{2^{m-1}})(t)\\
&\leq 2^{2}y^{2^{m-1}}(t+\nu h)y^{2^{m-2}}(t+\nu h) (_{a}\Delta_{h}^{\nu}y^{2^{m-2}})(t)\leq\cdot\cdot\cdot\\
&\leq 2^{m}y^{2^{m-1}}(t+\nu h)y^{2^{m-2}}(t+\nu h)\cdot\cdot\cdot y^{2^{0}}(t+\nu h)(_{a}\Delta_{h}^{\nu}y^{2^{0}})(t),
\end{split}
\end{equation*}
which is equivalent to
\begin{equation*}
(_{a}\Delta_{h}^{\nu}y^{2^{m}})(t)\leq 2^{m}y^{(2^{m}-1)}(t+\nu h) (_{a}\Delta_{h}^{\nu}y)(t).
\end{equation*}
The proof is complete.
\end{proof}

\begin{theorem}\label{t4.2}
Assume $x=0$ is an equilibrium point of the system \eqref{3.2}.
Then the following statements are satisfied:

(i) For $x_{i}(t)\geq 0$ $(i=1,2,\cdots,n)$, $t\in(h\N)_{a}$, and $l\in\{2k+1,k\in\N_{1}\}$, if
the following condition is satisfied
\begin{equation*}
x_{i}^{l-1}(t+\nu h)f_{i}(t,x(t+vh))\leq 0,
\ \ t\in(h\N)_{a+(1-\nu)h},
\end{equation*}
then the system \eqref{3.2} is stable.
Also, if
\begin{equation*}
x_{i}^{l-1}(t+\nu h)f_{i}(t,x(t+vh))< 0,
\ \ t\in(h\N)_{a+(1-\nu)h}, \forall x_{i}\neq 0,
\end{equation*}
then the system \eqref{3.2} is asymptotically stable.

(ii) For $m\in\N_{1}$, if the following condition is satisfied
\begin{equation*}
x_{i}^{(2^{m}-1)}(t+\nu h)f_{i}(t,x(t+vh))\leq 0,\ \ t\in(h\N)_{a+(1-\nu)h},
\end{equation*}
then the system \eqref{3.2} is stable.
And if
\begin{equation*}
x_{i}^{(2^{m}-1)}(t+\nu h)f_{i}(t,x(t+vh))< 0,\ \ t\in(h\N)_{a+(1-\nu)h}, \forall x_{i}\neq 0,
\end{equation*}
then the system \eqref{3.2} is asymptotically stable.
\end{theorem}

\begin{proof}
(i) Let us propose the following Lyapunov function,
 which is positive definite
 \begin{equation*}
 V(t,x(t))=\sum_{i=1}^{n}\frac{x_{i}^{l}(t)}{l}.
 \end{equation*}
Using Proposition \ref{l4.3} gives us
\begin{equation*}
(_{a}\Delta_{h}^{\nu}V)(t)\leq \sum_{i=1}^{n}x_{i}^{l-1}(t+\nu h)(_{a}\Delta_{h}^{\nu}x_{i})(t)
=\sum_{i=1}^{n}x_{i}^{l-1}(t+\nu h)f_{i}(t,x(t+\nu h))\leq 0.
\end{equation*}
By Lemma \ref{l3.6}, we have
\begin{equation*}
V(t,x(t))\leq V(a,x(a)),
\end{equation*}
that is,
\begin{equation*}
\sum_{i=1}^{n}\frac{x_{i}^{l}(t)}{l}\leq
\sum_{i=1}^{n}\frac{x_{i}^{l}(a)}{l}.
\end{equation*}
According to the definition of stability in the sense of Lyapunov,
we obtain the system \eqref{3.2} is stable in the sense of Lyapunov.

If
\begin{equation*}
x_{i}^{l-1}(t+\nu h) f_{i}(t,x(t+vh))< 0,
\ \ t\in(h\N)_{a+(1-\nu)h}, \forall x_{i}\neq 0,
\end{equation*}
similar to the above step, we can show that
the system \eqref{3.2} is stable.
By Proposition \ref{l4.3}, we have $(_{a}\Delta_{h}^{\nu}V)(t,x(t))\leq \sum_{i=1}^{n}x_{i}^{l-1}(t+\nu h)(_{a}\Delta_{h}^{\nu}x_{i})(t)< 0$,
that is, the fractional order $h$-difference of
$V$ function is negative definite.
Given the relationship between positive definite functions and class-$\mathcal{K}$ functions in \cite{sl1991}.
It can be concluded from Lemma \ref{l3.9} that the system \eqref{3.2} is asymptotically stable.

(ii) The proof is similar to the previous one,
by Proposition \ref{l4.4} and the positive definite
Lyapunov function:
\begin{equation*}
V(t,x(t))=\sum_{i=1}^{n}\frac{x_{i}^{2^{m}}(t)}{2^{m}}.
\end{equation*}
The proof is complete.
\end{proof}

\section{Numerical Results}
Now, we give some numerical examples to illustrate the application of the results established in the previous sections.
\begin{example}\label{e5.1}
Consider the following fractional order $h$-difference system
\begin{equation}\label{5.1}
\left\{
\begin{array}{ll}
(_{a}\Delta _{h,\ast}^{\nu}x_{1})(t)=-x_{1}(t+\nu h),\ \ x_{1}(a)=0.1,\\
(_{a}\Delta _{h,\ast}^{\nu}x_{2})(t)=-x_{2}(t+\nu h),\ \ x_{2}(a)=0.2,
\end{array}
\right.
\end{equation}
where $\nu=0.5$, 
$a=0$, $h=1$, $t\in (h\N)_{a+(1-\nu)h}$, and
this difference system has a trivial solution $x(t)=(x_{1}(t),x_{2}(t))^{T}=0$.

We can see that
\begin{equation*}
\begin{split}
x^{T}(t+\nu h)P f(t,x(t+vh))
&=(x_{1}(t+\nu h),x_{2}(t+\nu h))P\left[ {\begin{array}{*{20}{c}}
-x_{1}(t+\nu h)\\
-x_{2}(t+\nu h)
\end{array}} \right]
=-(x_{1}(t+\nu h)+x_{2}(t+\nu h))^{2}\leq 0,
\end{split}
\end{equation*}
where
$P=\left[ {\begin{array}{*{20}{c}}
1&1\\
1&1
\end{array}} \right]$.

Thus, from Theorem \ref{t3.1}, the origin of the system \eqref{5.1} is stable, as it can be seen from Figures \ref{fig:1} and \ref{fig:2}.
\begin{figure}
\begin{minipage}{65mm}
\includegraphics[width=\linewidth,height=\linewidth]{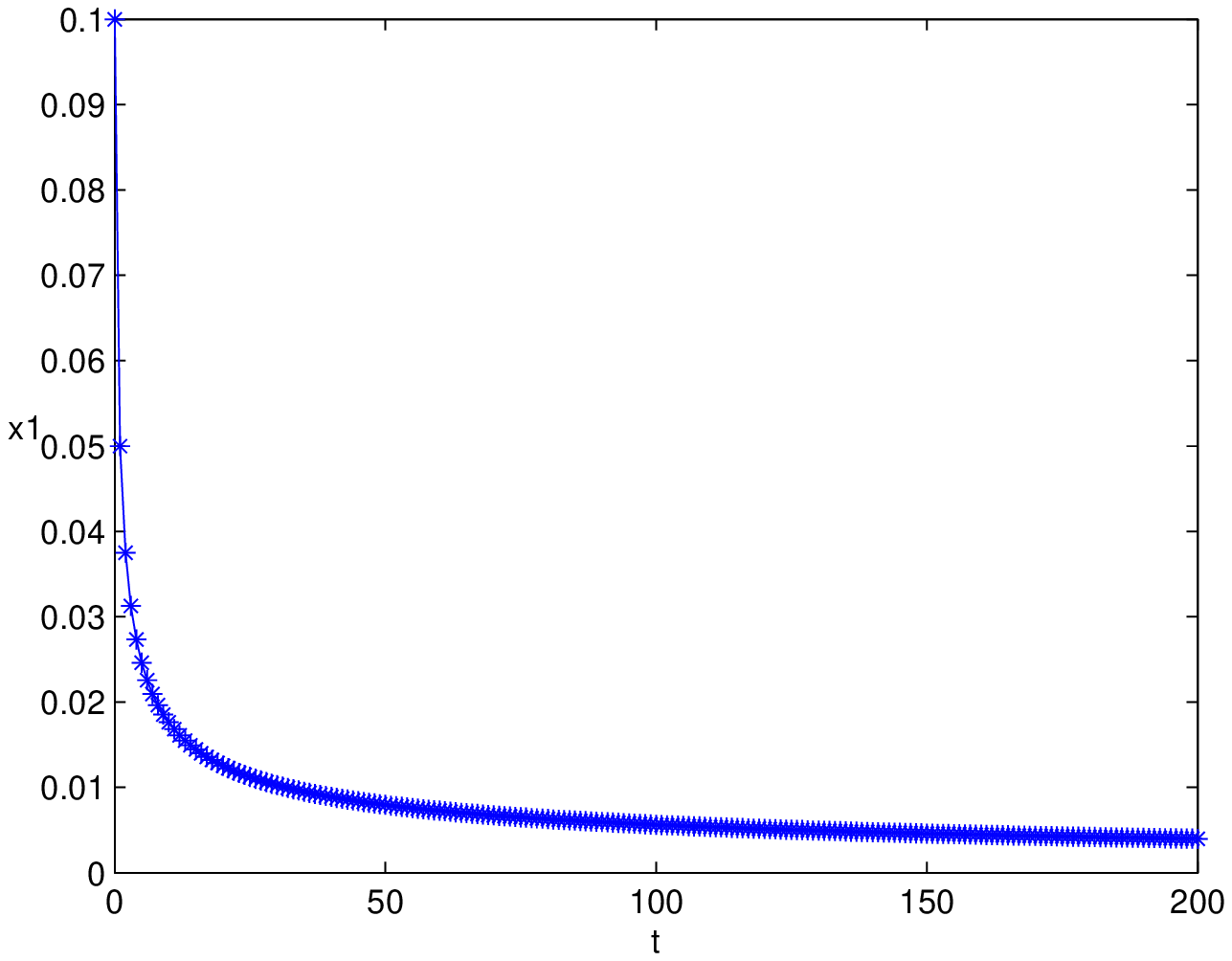}
\caption{Stability of $x_{1}$ for $\nu=0.5$.}
\label{fig:1}
\end{minipage}
\hfil
\begin{minipage}{65mm}
\includegraphics[width=\linewidth,height=\linewidth]{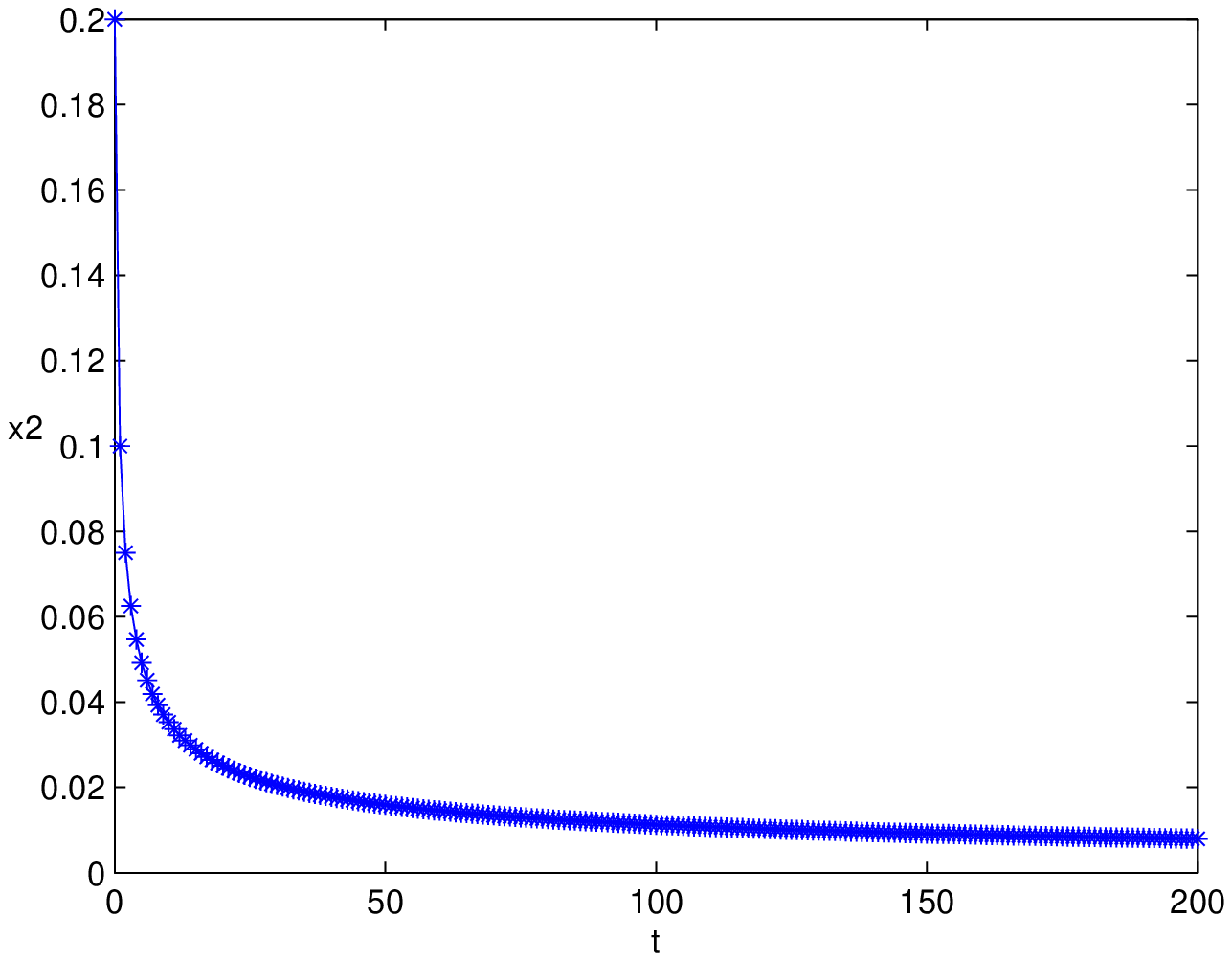}
\caption{Stability of $x_{2}$ for $\nu=0.5$.}
\label{fig:2}
\end{minipage}
\end{figure}
\end{example}

\begin{example}\label{e5.2}
Consider the following fractional order $h$-difference system
\begin{equation}\label{5.2}
\left\{
\begin{array}{ll}
(_{a}\Delta _{h}^{\nu}x_{1})(t)=-\frac{1}{2}x_{2}^{16}(t+\nu h)x_{1}(t+\nu h),\ \ x_{1}(a)=0.1,\\
(_{a}\Delta _{h}^{\nu}x_{2})(t)=-\frac{1}{2}x_{1}^{2}(t+\nu h)x_{2}(t+\nu h),\ \ x_{2}(a)=0.2,
\end{array}
\right.
\end{equation}
where $\nu=0.5$, 
$a=0$, $h=1$, $t\in (h\N)_{a+(1-\nu)h}$, and
this difference system has a trivial solution $x(t)=(x_{1}(t),x_{2}(t))^{T}=0$.

We can see that
\begin{equation*}
\begin{split}
x^{T}(t+\nu h)P f(t,x(t+vh))
&=(x_{1}(t+\nu h),x_{2}(t+\nu h))P\left[ {\begin{array}{*{20}{c}}
-\frac{1}{2}x_{2}^{16}(t+\nu h)x_{1}(t+\nu h)\\
-\frac{1}{2}x_{1}^{2}(t+\nu h)x_{2}(t+\nu h)
\end{array}} \right]\\
&=-\frac{1}{2}x_{2}^{16}(t+\nu h)x_{1}^{2}(t+\nu h)
-\frac{1}{2}x_{2}^{2}(t+\nu h)x_{1}^{2}(t+\nu h)\leq 0,
\end{split}
\end{equation*}
where
$P=\left[ {\begin{array}{*{20}{c}}
1&0\\
0&1
\end{array}} \right]$.

Thus, from Theorem \ref{t3.2}, the origin of
the system \eqref{5.2} is stable,
as it can be seen from Figures \ref{fig:3} and \ref{fig:4}.
\begin{figure}
\begin{minipage}{65mm}
\includegraphics[width=\linewidth,height=\linewidth]{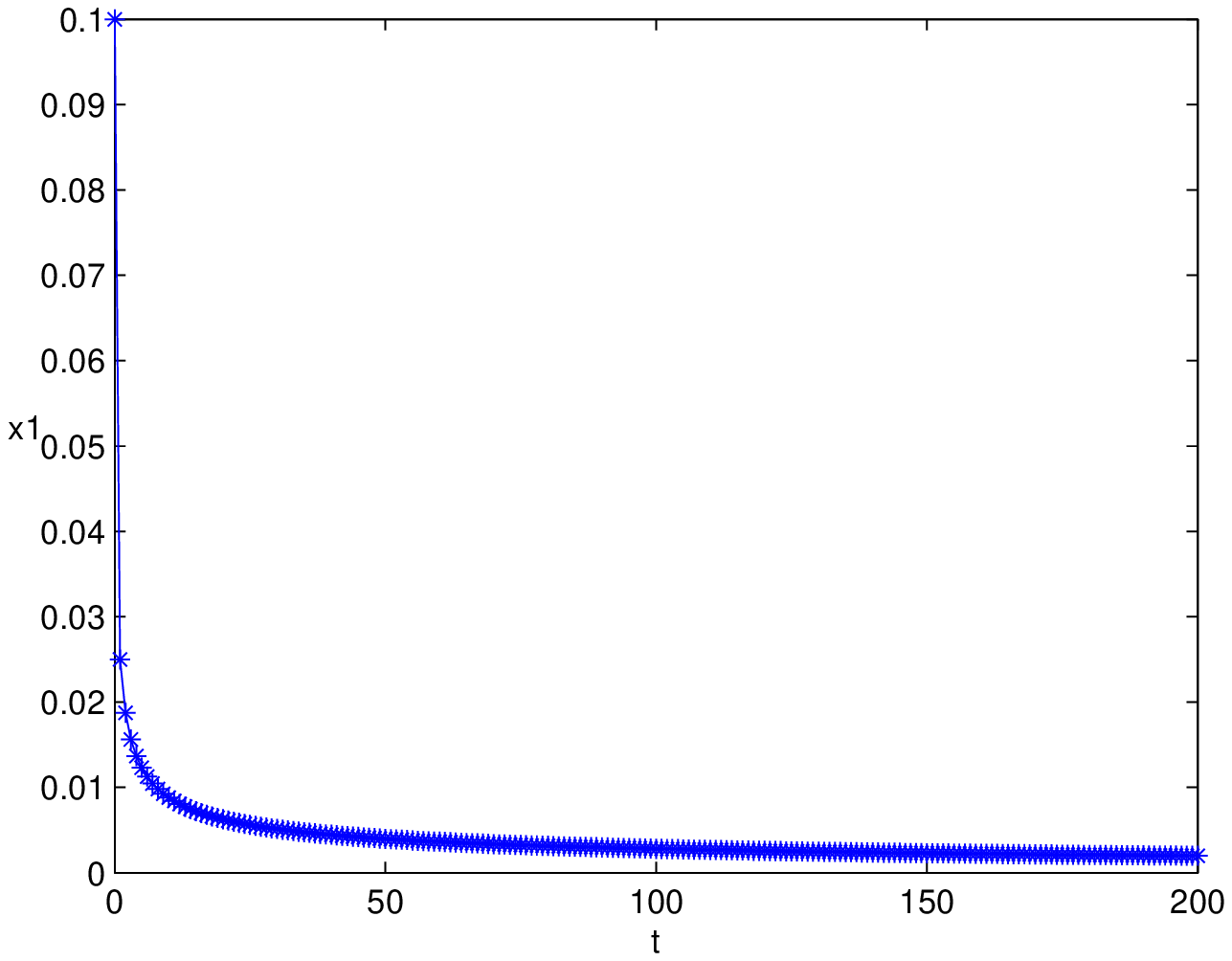}
\caption{Stability of $x_{1}$ for $\nu=0.5$.}
\label{fig:3}
\end{minipage}
\hfil
\begin{minipage}{65mm}
\includegraphics[width=\linewidth,height=\linewidth]{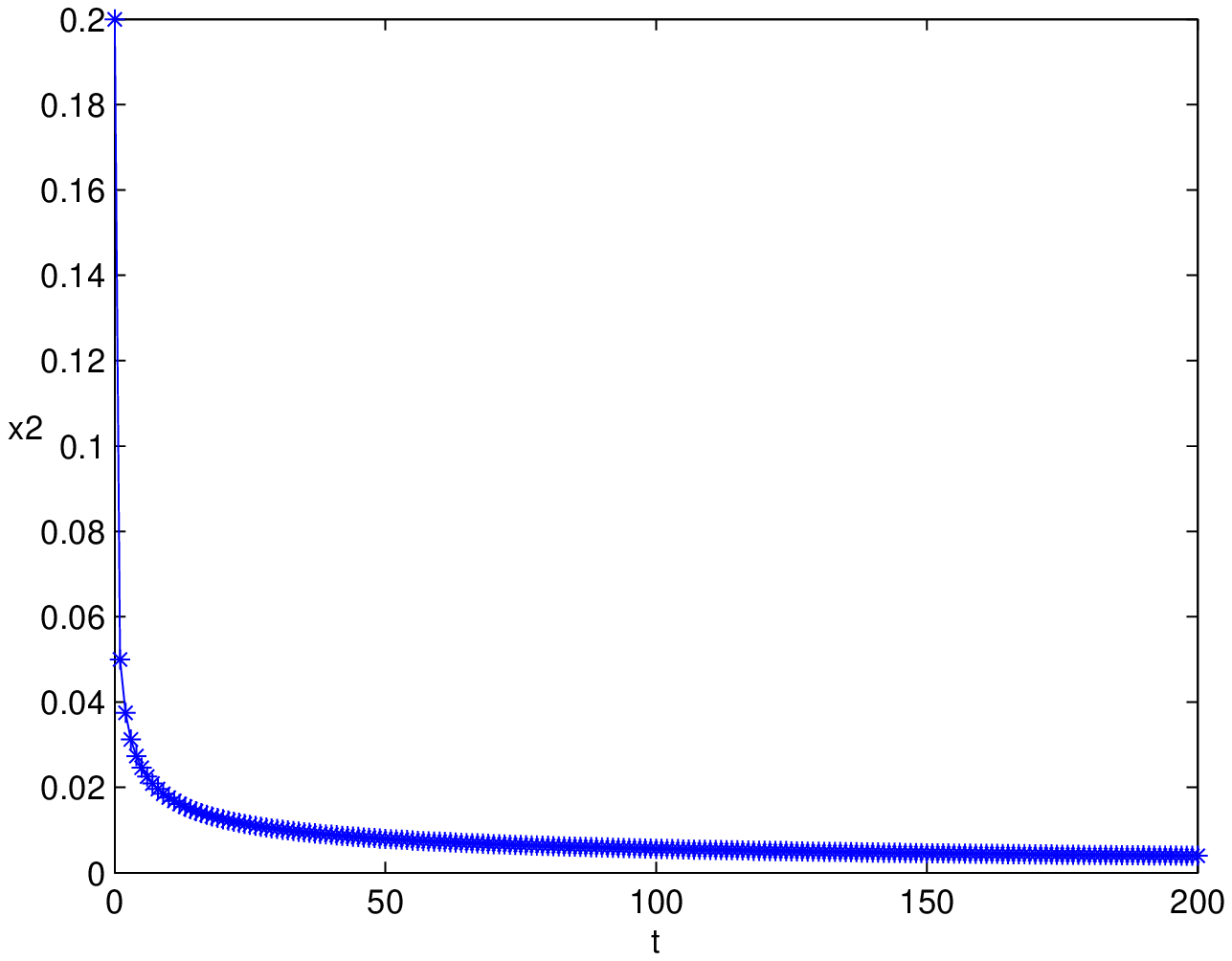}
\caption{Stability of $x_{2}$ for $\nu=0.5$.}
\label{fig:4}
\end{minipage}
\end{figure}
\end{example}

\begin{example}\label{e5.3}
Consider the following fractional order $h$-difference system
\begin{equation}\label{5.3}
\left\{
  \begin{array}{ll}
(_{a}\Delta _{h,\ast}^{\nu}x_{1})(t)=-x_{1}^{3}(t+\nu h),\ \
x_{1}(a)=0.4,\\
(_{a}\Delta _{h,\ast}^{\nu}x_{2})(t)=-x_{1}^{2}(t+\nu h)-x_{2}(t+\nu h),\ \ x_{2}(a)=0.2,
\end{array}
\right.
\end{equation}
where $\nu=0.5$, $x_{i}(t)\geq0$ $(i=1,2)$, $a=0$, $h=1$, $t\in (h\N)_{a+(1-\nu)h}$, and
this difference system has a trivial solution
$x(t)=(x_{1}(t),x_{2}(t))^{T}=0$.

We can see that
\begin{equation*}
\begin{split}
x_{1}^{2}(t+\nu h) (_{a}\Delta _{h,\ast}^{\nu}x_{1})(t)
&=x_{1}^{2}(t+\nu h)(-x_{1}^{3}(t+\nu h))
=-x_{1}^{5}(t+\nu h)\leq 0,
\end{split}
\end{equation*}
\begin{equation*}
\begin{split}
x_{2}^{2}(t+\nu h) (_{a}\Delta _{h,\ast}^{\nu}x_{2})(t)
&=x_{2}^{2}(t+\nu h)(-x_{1}^{2}(t+\nu h)-x_{2}(t+\nu h))
=-x_{1}^{2}(t+\nu h)x_{2}^{2}(t+\nu h)
-x_{2}^{3}(t+\nu h)
\leq 0.
\end{split}
\end{equation*}
Thus, from Theorem \ref{t4.1} (i), the origin of
the system \eqref{5.3} is stable,
as it can be seen from Figures \ref{fig:5} and \ref{fig:6}.
\begin{figure}
\begin{minipage}{65mm}
\includegraphics[width=\linewidth,height=\linewidth]{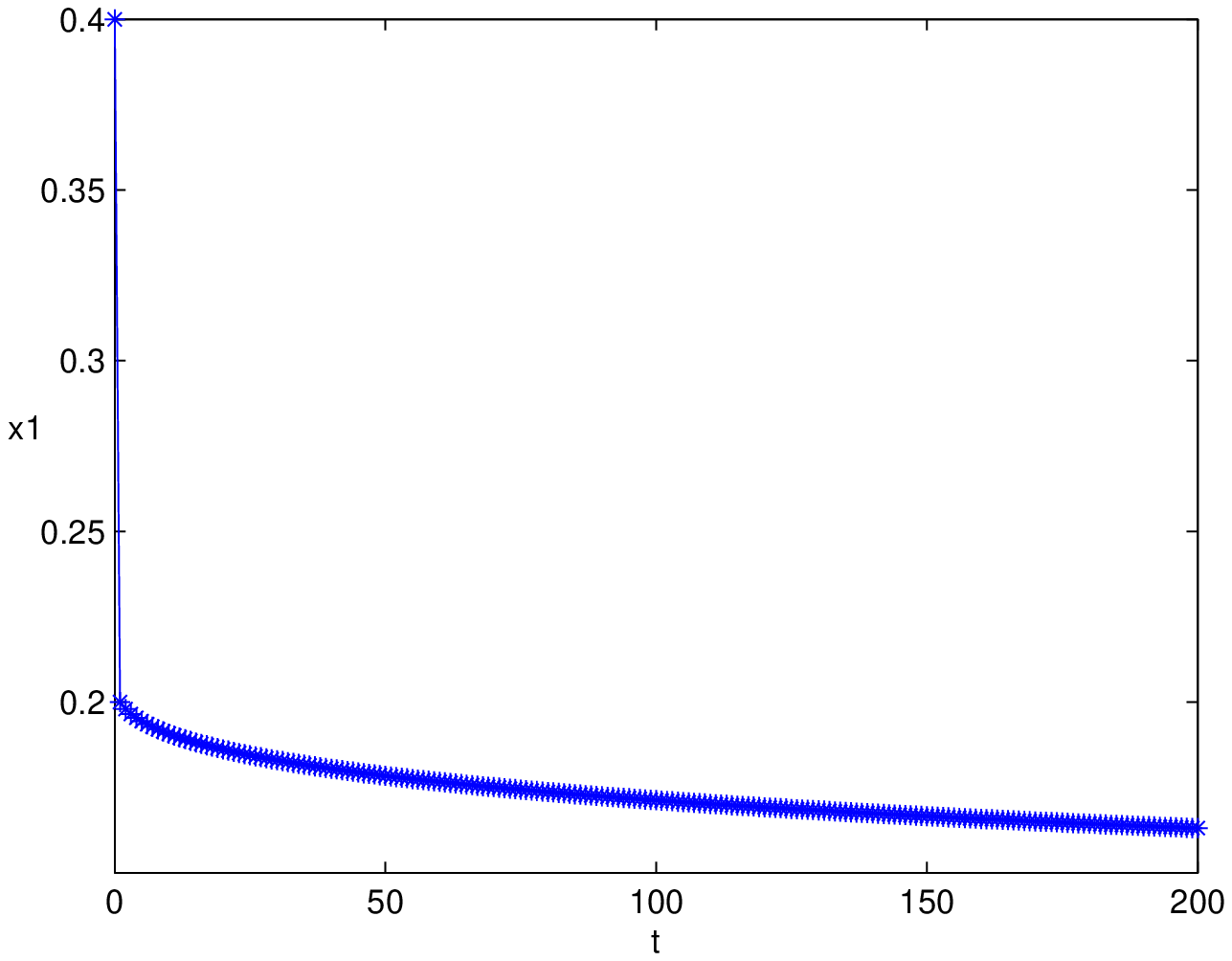}
\caption{Stability of $x$ for $\nu=0.5$.}
\label{fig:5}
\end{minipage}
\hfil
\begin{minipage}{65mm}
\includegraphics[width=\linewidth,height=\linewidth]{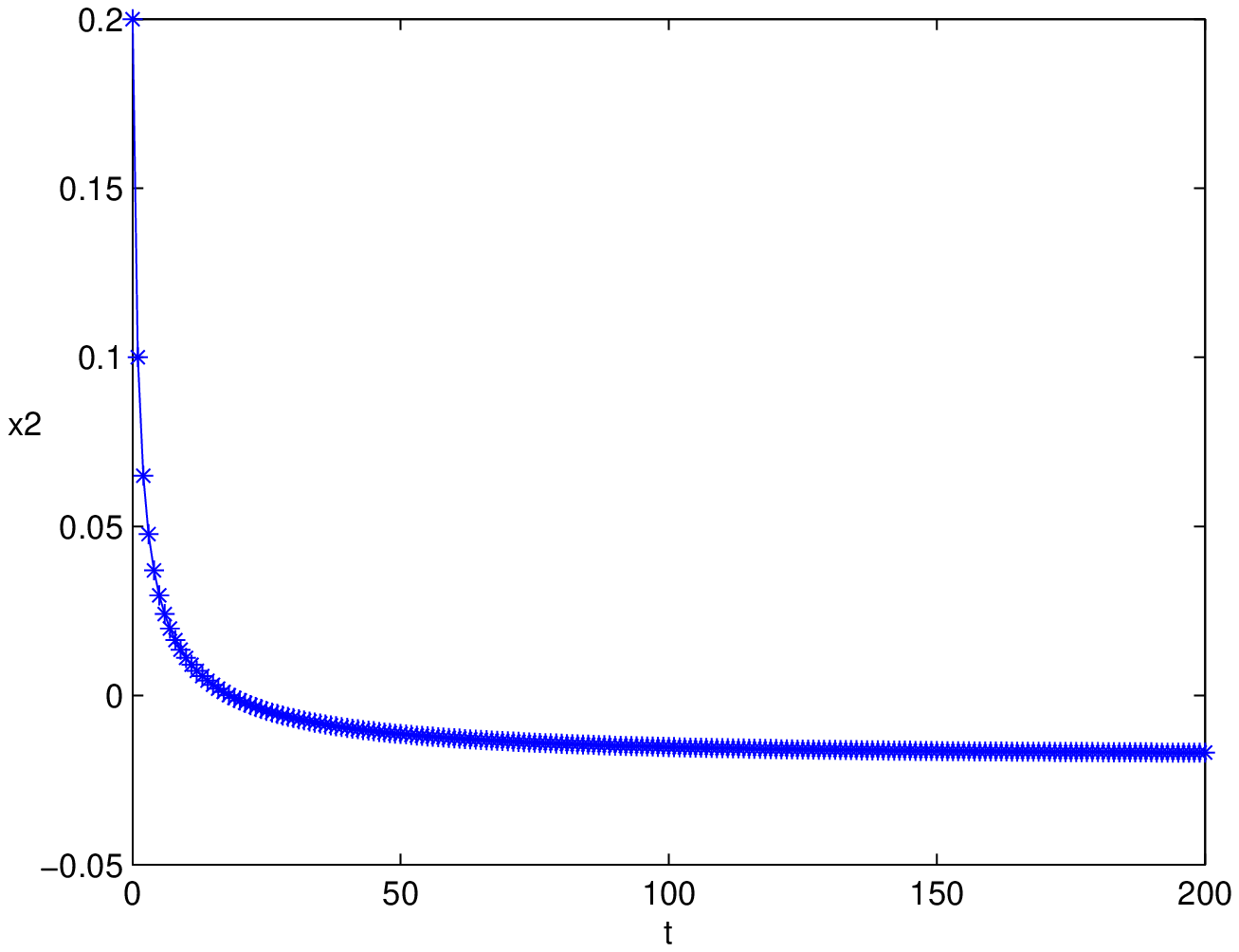}
\caption{Stability of $x$ for $\nu=0.5$.}
\label{fig:6}
\end{minipage}
\end{figure}
\end{example}

\begin{example}\label{e5.3}
Consider the following fractional order $h$-difference equation
\begin{equation}\label{5.4}
\left\{
  \begin{array}{ll}
(_{a}\Delta _{h}^{\nu}x_{1})(t)=-x_{1}(t+\nu h)-x_{2}^{3}(t+\nu h),\ \
(_{a}\Delta _{h}^{\nu-1}x_{1})(t)|_{t=a+(1-\nu)h}=h^{1-\nu}0.4,\\
(_{a}\Delta _{h,\ast}^{\nu}x_{2})(t)=-x_{1}^{2}(t+\nu h),
\ \ (_{a}\Delta _{h}^{\nu-1}x_{2})(t)|_{t=a+(1-\nu)h}=h^{1-\nu}0.2,
\end{array}
\right.
\end{equation}
where $\nu=0.5$, $x_{i}(t)\geq0$ $(i=1,2)$, $a=0$, $h=1$, $t\in (h\N)_{a+(1-\nu)h}$, and
this difference equation has a trivial solution
$x(t)=(x_{1}(t),x_{2}(t))^{T}=0$.

We can see that
\begin{equation*}
\begin{split}
x_{1}^{2}(t+\nu h) (_{a}\Delta _{h}^{\nu}x_{1})(t)
&=x_{1}^{2}(t+\nu h)(-x_{1}(t+\nu h)-x_{2}^{3}(t+\nu h))
=-x_{1}^{3}(t+\nu h)-x_{1}^{2}(t+\nu h)x_{2}^{3}(t+\nu h)\leq 0,
\end{split}
\end{equation*}
\begin{equation*}
\begin{split}
x_{2}^{2}(t+\nu h) (_{a}\Delta _{h}^{\nu}x_{2})(t)
&=x_{2}^{2}(t+\nu h)(-x_{1}^{2}(t+\nu h))
=-x_{1}^{2}(t+\nu h)x_{2}^{2}(t+\nu h)\leq 0.
\end{split}
\end{equation*}
Thus, from Theorem \ref{t4.2} (i),
the origin of the equation \eqref{5.4} is stable,
as can be seen from Figures \ref{fig:7} and \ref{fig:8}.
\begin{figure}
\begin{minipage}{65mm}
\includegraphics[width=\linewidth,height=\linewidth]{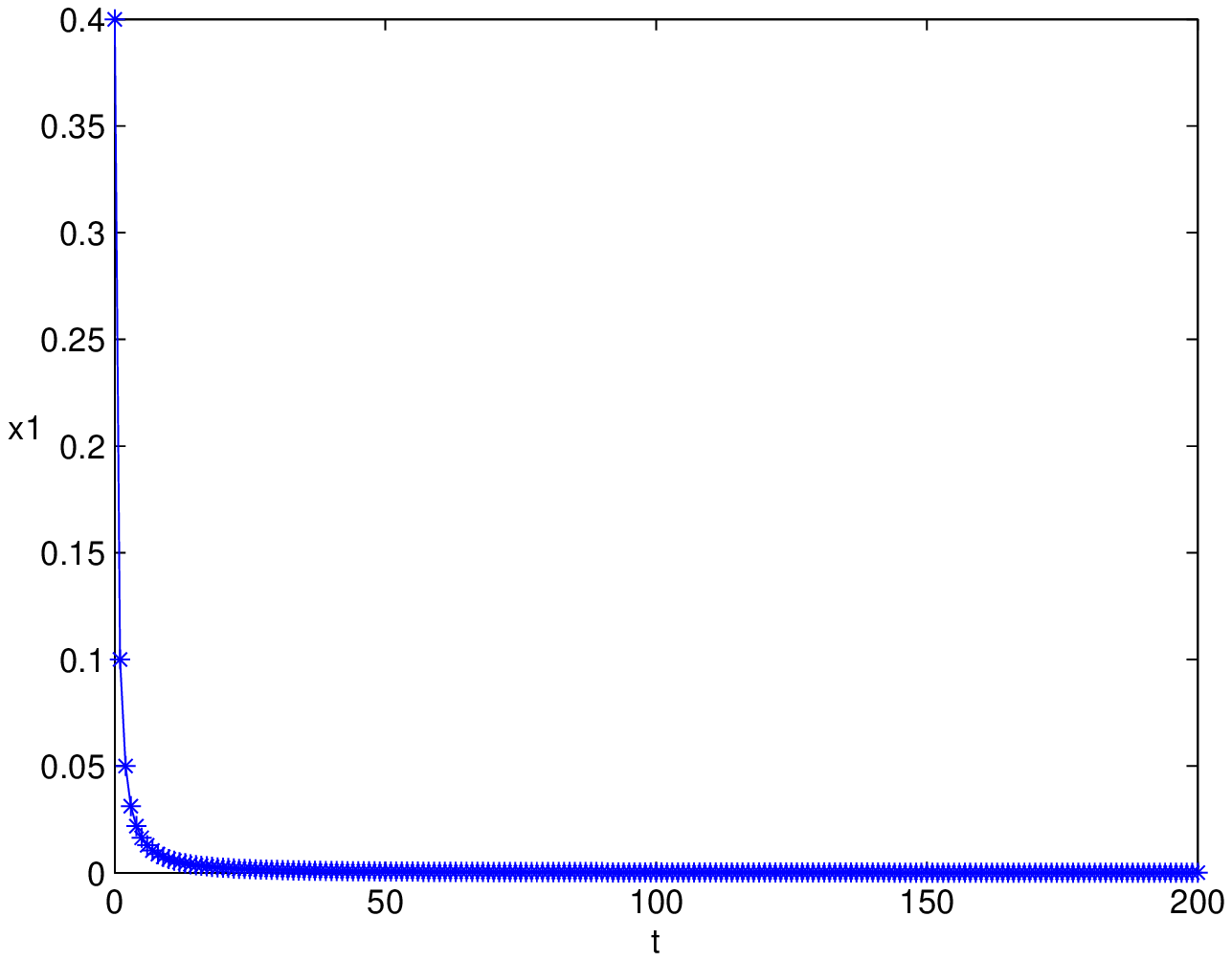}
\caption{Stability of $x$ for $\nu=0.5$.}
\label{fig:7}
\end{minipage}
\hfil
\begin{minipage}{65mm}
\includegraphics[width=\linewidth,height=\linewidth]{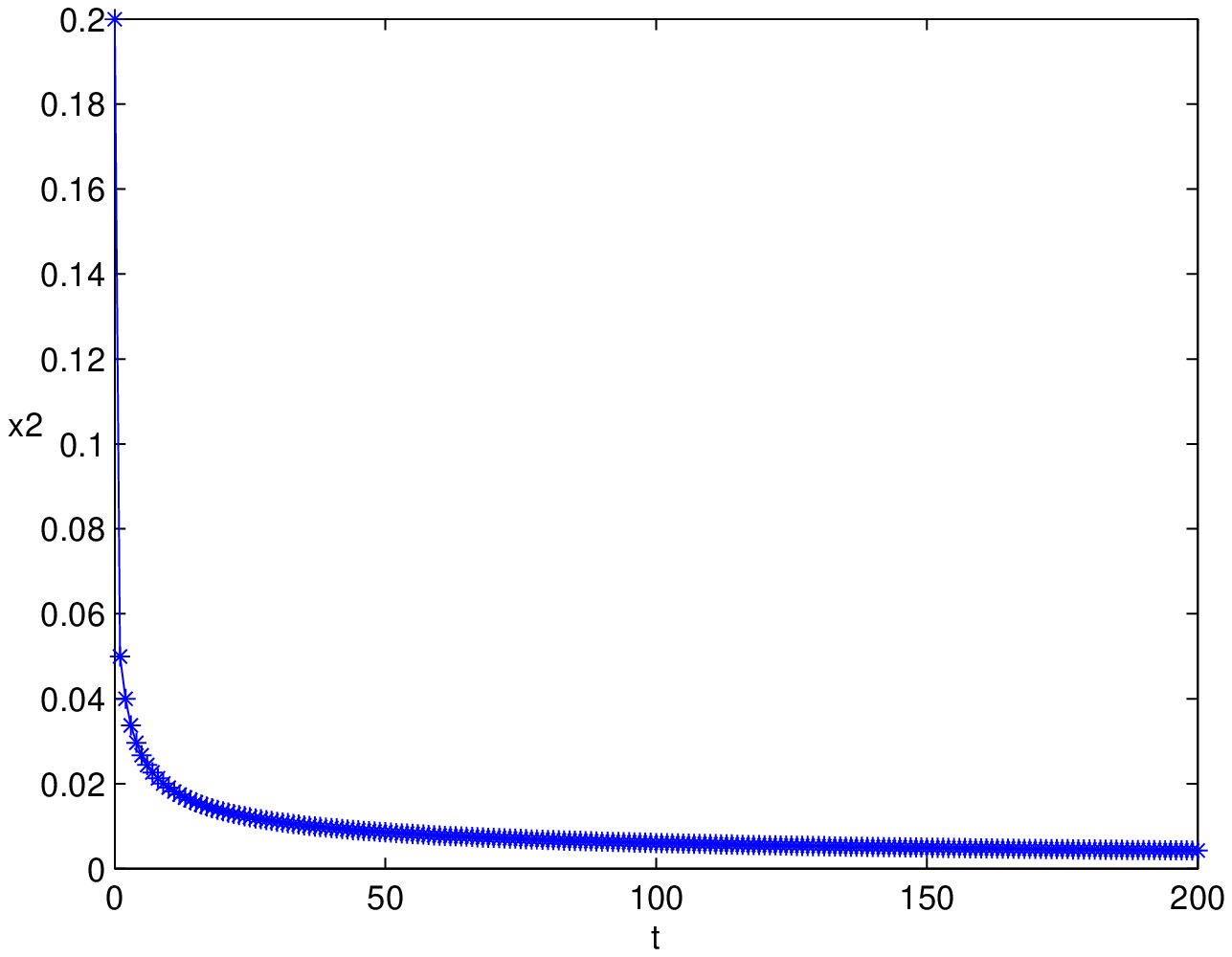}
\caption{Stability of $x$ for $\nu=0.5$.}
\label{fig:8}
\end{minipage}
\end{figure}
\end{example}
\section{Conclusion}

This paper presents some new propositions, which allow the application of general quadratic Lyapunov functions to the stability analysis of the fractional order $h$-difference systems by means of the discrete fractional Lyapunov direct method. In addition, this work  gives a generalization of Lemma 2.10 in \cite{wb2017} and Lemma 3.2 in \cite{bw2017}, that allows establishing a broader family of
Lyapunov functions to determine the stability of the fractional order
$h$-difference systems. As a result, we give the sufficient conditions
for these systems to be stable or asymptotically stable. In addition,
some examples are given to show the established results.
\section*{Acknowledgements}
The authors would like to thank the anonymous reviewers
for their valuable comments and suggestions, which improved 
the quality of the paper.


\end{document}